\documentclass[12pt,a4paper]{article}

\usepackage[margin=1in]{geometry}
\usepackage{tikz,amsmath,amsthm,amssymb}
\usetikzlibrary{calc}
\usepackage{enumitem}
\usepackage{nicefrac}
\usepackage[mathscr]{euscript}
\usepackage{comment,xspace}
\usepackage[colorlinks=true,citecolor=black,linkcolor=black,urlcolor=blue]{hyperref}
\newcommand{\arxiv}[1]{\href{http://arxiv.org/abs/#1}{\texttt{arXiv:#1}}}

\theoremstyle{plain}
\newtheorem{lem}{Lemma}[section]
\newtheorem{thm}{Theorem}[section]
\newtheorem{cor}{Corollary}[section]
\theoremstyle{remark}
\newtheorem{rem}{Remark}

\def\nfrac#1#2{{\textstyle\frac{#1}{#2}}}

\newcommand{\cM}{\mathcal{M}}
\newcommand{\cN}{\mathcal{N}}
\newcommand{\cX}{\mathcal{X}}
\newcommand{\dtv}{\mathrm{d}_{\mathrm{TV}}}
\newcommand{\bbR}{\mathbb{R}}
\newcommand{\cG}{\mathcal{G}}

\newcommand{\Gnd}[1][n,\dsq]{\mathcal{G}_{#1}}

\newcommand{\dsq}{\boldsymbol{d}}

\newcommand{\pois}[1]{\ensuremath{\textrm{Pois}(#1)}}
\newcommand{\law}[1]{\ensuremath{\mathcal{L}(#1)}}
\newcommand{\distas}{\thicksim_d}
\newcommand{\distapp}{\thickapprox_d}
\newcommand{\Expn}{\mathbb{E}}
\newcommand{\Ex}[1]{\mathbb{E}_{\mathrm{#1}}}

\newcommand{\ds}{$\triangle$-switch}
\newcommand{\dsp}{$\triangle^+$-switch}
\newcommand{\dsm}{$\triangle^-$-switch}
\newcommand{\Dp}{$\triangle^+$}
\newcommand{\Dm}{$\triangle^-$}
\newcommand{\bard}{\bar{d}}
\newcommand{\wlg}{without loss of generality}
\newcommand{\ifff}{if and only if\xspace}
\newcommand{\Nb}{\mbox{N}}

\newcommand{\sign}{\textrm{sign}}
\newcommand{\norm}[1]{\|#1\|}
\newcommand{\unif}{\ensuremath{\pi'}\xspace}
\newcommand{\mugeneral}{\beta}   

\newcommand{\es}{\emptyset}
\newcommand{\sm}{\setminus}

\tikzset{ b/.style = { circle 
                     , draw
                     , thick
                     , inner sep = 0pt
                     , fill = black
                     , minimum size = 2.5pt
                     }}

\tikzset{every picture/.style={line width=0.5pt}}

\date{15 June 2025}

\title{Triangle processes on graphs\\ with given degree sequence}
\author{%
Colin Cooper\\
{\small Department of Informatics}\\[-0.5ex]
{\small King's College}\\[-0.5ex]
{\small London WC2B 4BG, U.K.}\\[-0.2ex]
{\small \texttt{colin.cooper@kcl.ac.uk}}\\[-0.5ex]
{\small \texttt{orcid:\,0000-0002-5264-4401}}\\
\and
Martin Dyer\\
{\small School of Computing}\\[-0.5ex]
{\small University of Leeds}\\[-0.5ex]
{\small Leeds LS2 9JT, U.K.}\\[-0.2ex]
{\small \texttt{m.e.dyer@leeds.ac.uk}}\\[-0.5ex]
{\small \texttt{orcid:\,0000-0002-2018-0374}}\\
\and
Catherine Greenhill\\
{\small School of Mathematics and Statistics}\\[-0.5ex]
{\small UNSW Sydney}\\[-0.5ex]
{\small NSW 2052, Australia}\\[-0.2ex]
{\small \texttt{c.greenhill@unsw.edu.au}}\\[-0.5ex]
{\small \texttt{orcid:\,0000-0001-6998-2282}}\\
}

\begin{document}

\maketitle

\begin{abstract}
The switch chain is a well-studied Markov chain which generates random graphs with a given degree sequence and has uniform stationary distribution.
Motivated by the high number of triangles seen in some real-world networks,
we study a variant of the switch chain which is more likely to produce graphs 
with higher numbers
of triangles. Specifically, we apply a Metropolis scheme
designed to have the following stationary
distribution: graph $G$ has probability proportional to $\lambda^{\min\{t(G),\nu\}}$,
where $t(G)$ is the number of triangles in $G$ and $\nu$ is a cut-off value introduced 
to moderate the impact of graphs with a very high number of triangles.  We assume that
the ``activity'' $\lambda$ satisfies $\lambda\geq 1$, and call the resulting chain the modified Metropolis switch chain.
We prove that the modified Metropolis switch chain is rapidly mixing whenever
the (standard) switch
chain is rapidly mixing, provided that the activity and maximum degree are not too large.

The triangle switch (or ``\ds'') chain is a restriction of the switch chain which
only performs switches that change the set of triangles in the graph.
We prove that the \ds\ chain is irreducible for any degree sequence with minimum degree at least~3, and prove a rapid mixing result for the modified
Metropolis \ds\ chain. 

Finally, we investigate the distribution of triangles in random graphs with given degrees, under both the uniform distribution and the distribution in 
which graph $G$ has probability proportional to $\lambda^{t(G)}$. 
Our analysis implies that the imposition of the cut-off $\nu$ does not significantly impact the behaviour of these modified Metropolis chains over polynomially many steps.
\end{abstract}

\abovedisplayskip=1ex \belowdisplayskip=1ex

\section{Introduction}\label{sec:intro}
Randomly generating graphs has been the subject of considerable research interest.
See, for example,~\cite{AMPTW,ABLMO,AK,BKS,CDG07,CDGH,FGMS,PGW,KTV,MES,MS,TikYou}.
A comprehensive survey of work to date is given in~\cite{Greenhill}.

Generation using Markov chains has been an important approach in this area, particularly Markov chains based on switches, for example~\cite{AK,CDG07,CDGH,FGMS,KTV,LMM,MES,MS,TikYou}. Switches delete a pair of edges from the graph and insert a pair using the same four vertices. They have the important property that they preserve the degree sequence of the graph. Thus they are useful for generating regular graphs, or other graphs with a given degree sequence. Markov chains also give a dynamic reconfigurability property, which is useful in some applications, for example~\cite{CDG07,FGMS,MS}.

For any such Markov chain, three questions arise. Firstly, can the chain generate any graph in the chosen class?  In other words, is the Markov chain ergodic? Secondly, if so what is its equilibrium distribution? Thirdly, what is the rate of convergence to equilibrium? That is, what is the mixing time of the chain? Note that it is usually fairly easy to ensure that a Markov chain is aperiodic, so the second question focusses on establishing irreducibility.

The expected number of triangles in a uniformly random graph with a given degree
sequence can be bounded above by the cube of the maximum degree.  Hence, for example,
when all
degrees are bounded we expect only a bounded number of triangles in this random graph.
However, many real-world networks, such as social networks, contain many 
triangles~\cite{GKM,JGN}.  This motivated us to adapt the transition procedure
of the switch chain by introducing a Metropolis accept/reject scheme~\cite{hastings,metropolis}, described
in Section~\ref{sec:chain}.  
This chain has an equilibrium distribution $\hat{\pi}_\lambda$ in which the probability of a graph is exponentially weighted by its number of triangles, up to some cut-off which is required to moderate the impact of graphs with a very high number of triangles. 
The chain has a parameter $\lambda\geq 1$ which is a fixed real number.
Our first result (Theorem~\ref{thm:Metropolis-switch}) shows that this chain, which we call the \emph{modified Metropolis switch chain},
is rapidly mixing whenever the (standard) switch chain is rapidly mixing,
under some conditions on the degree sequence and the parameter $\lambda$ of the chain.

The \emph{triangle switch Markov chain}, which we abbreviate as ``\ds\ chain'', was introduced in~\cite{CDG19} to explore ways to generate regular random graphs with more triangles than would appear under the uniform distribution. 
See Section~\ref{sec:chain} below for the relevant definitions. 
Transitions of the \ds\ chain are precisely those switches which change the set of triangles in the graph; 
several options for the transition probabilities were proposed in~\cite{CDG19}.
These transition probabilities were designed to encourage the formation
of more triangles in fewer steps, compared with the standard switch chain.
One variant was shown to produce $\Omega(n)$ triangles in $O(n)$ steps of the chain~\cite[Theorem~2]{CDG19}, though it is not known whether that variant of the chain is close to stationarity 
after $O(n)$ steps.

The \ds\ chain was proved to be irreducible for 3-regular graphs (note, irreducibility depends only on the set of transitions with positive probability, and not the values of the transition probabilities). 
The proof of irreducibility involved finding a sequence of transitions of the chain (\ds es) which transform any 3-regular graph into disjoint copies of $K_4$ together with a disjoint 6-vertex graph, and then showing that any such graph can be transformed into another using \ds es. This proof is not well suited to standard methods for proving rapid mixing of the chain. A further obstacle is that some of the variants studied in~\cite{CDG19} do not seem to be time-reversible, a requirement for most approaches to proving rapid mixing.

The proof of irreducibility in~\cite{CDG19} was extended in~\cite{CDG21} to $d$-regular graphs, for any $d$. The proof involved finding a sequence of \ds es which  transform any $d$-regular graph into disjoint copies of $K_{d+1}$ together with a disjoint graph, called a ``fragment'', with at most $2d$ vertices; then showing that any such graph can be transformed into another using \ds es. This required a different approach to proving irreducibility of the fragments, by showing that any step of the flip chain of~\cite{MS} can be simulated by at most three \ds es.

Our second result (Theorem~\ref{thm:irreducible}) extends and generalises the result of~\cite{CDG19} by showing that any step of the switch chain~\cite{CDG07} can be simulated by at most five \ds es. Moreover this holds for all graphs with a given degree sequence such that the minimum degree is at least~3. 

Next we apply the Metropolis accept/reject approach to the \ds\ chain, producing
a Markov chain which we called the \emph{modified Metropolis \ds\ chain}.
This chain has the same stationary distribution $\hat{\pi}_\lambda$ as the modified Metropolis switch
chain.  Our third result (Theorem~\ref{thm:rapid}) analyses the mixing time of the modified Metropolis \ds\ chain, 
using the simulation paths that we introduced to
prove irreducibility.

Finally, we analyse the number of triangles in a randomly chosen graph with
given degree sequence, under two distributions: the uniform distribution, and
the Gibbs distribution $\pi_\lambda$ which weights each graph exponentially
by the its number of triangles (without any cut-off).  
This analysis, presented in Section~\ref{sec:distribution}, 
may be of independent interest.  We use it to show that the introduction of
the cut-off (which limits the impact of graphs with a very high number of triangles)
does not noticeably affect the performance of the chain over polynomial-length runtimes.

Our main results are presented formally in Section~\ref{sec:our-results}, where
we also describe the structure of the rest of the paper. 

We only know of a few other works which rigorously analyse Markov
chains for generating graphs from a known non-uniform distribution.
It may be possible that the algorithm of Jerrum, Sinclair and Vigoda~\cite{JSV} for the permanent could be used to sample bipartite graphs with given degree sequence 
from a non-uniform distribution by adjusting the edge weights, though to our knowledge
this has not been explored. Bhamidi, Bresler and Sly~\cite{BBS} analysed the
Glauber dynamics (single-edge update) for sampling from the exponential random graph model,
and proved rapid mixing in the ``high temperature regime'', but also proved
that the graphs generated in this regime are very close to binomial random
graphs $G(n,p)$, in the sense that any finite set of edges are asymptotically independent.
Recently, Jenssen, Perkins, Potukuchi and Simkin~\cite{JPPS} analysed the
Glauber dynamics for generating from $G(n,p)$ conditioned on being triangle-free,
and proved that this Markov chain is rapidly mixing when $p\leq cn^{-1/2}$
for some sufficiently small $c>0$.

Before proceeding, we mention some related work on sampling graphs with a given degree sequence
and additional constraints. Mahlmann and Schindelhauer~\cite{MS} proposed a restricted set of 
switches, called $k$-flips, which maintain connectivity. The 1-flip chain is rapidly
mixing on the set of $d$-regular connected graphs, for any $d$~\cite{CDGH,FGMS}.
In some applications it can be useful to specify the number of edges between sets of vertices 
with given degrees, as well as the degree sequence.  This is done with a \emph{joint degree
matrix}. Work on the switch chain in this context can be found in~\cite{AK-jd,CDEM,EMT,SP}.

\subsection{Notation and definitions}\label{sec:notation}
Let $[n]$ denote the set $\{1,2,\ldots, n\}$ and let $G=(V,E)$ be an $n$-vertex (simple) graph with vertex set $V=[n]$. The degree sequence $\dsq=(d_1,d_2,\ldots,d_n)$ of $G$ is given by $d_i=\deg(i) $ for all $i\in[n]$. An arbitrary sequence $\dsq$ of $n$ nonnegative integers is called \emph{graphical} if there is any graph which has degree sequence $\dsq$. Given the sequence, this property can be tested efficiently and constructively using, for example, the Havel-Hakimi algorithm~\cite{hakimi,havel}.

We assume \wlg\ that $d_1\geq d_2\geq\cdots\geq d_n$, so $d_1$ is the maximum degree and $d_n$ the minimum degree. We note that $d_1,d_n$ are usually denoted by $\Delta,\delta$, but we will not do so here. If $d_1=d_n=d$, the $d$-regular case, then $d$ is common degree of all vertices. Define
\[  M = M(\dsq) =\sum_{i=1}^n d_i=2|E|=n\bard, \qquad M_2 = M_2(\dsq) =\sum_{i=1}^n  d_i(d_i-1),\]
where $\bard$ is the \emph{average} degree.

We will denote the number of triangles in $G$ by
\[ t(G)=\big|\big\{\{i,j,k\}\subseteq V:ij,jk,ki\in E\big\}\big|.\]
If $G$ has degree sequence $\dsq$ then the maximum value for $t(G)$ is $\nicefrac13\sum_{i=1}^n\binom{d_i}{2}=M_2/6$. (The sum is over all ways of choosing two edges of the triangle meeting a given vertex, the divisor $\nicefrac13$ because this counts each triangle three times.)
In the regular case, this maximum is achieved by a
$d$-regular graph consisting of the disjoint union of
$n/(d+1)$ cliques of size $(d+1)$.

Let $\Gnd$ be the set of all labelled graphs with degree sequence $\dsq$.
The \emph{switch chain} is a Markov chain which walks on $\Gnd$ by choosing
two non-incident edges uniformly at random, and randomising these two edges without
changing the degree sequence. 
\begin{center}
\begin{tikzpicture}[scale=0.6]
\tikzset{empty/.style={rectangle,draw=none,fill=none}}
\draw [thick,-] (4,0) -- (2,0)  (2,2) -- (4,2);
\draw [dashed,thick] (4,0)--(4,2) (2,0)--(2,2);
\draw [fill] (2,0) circle (0.1); \draw [fill] (2,2) circle (0.1);
\draw [fill] (4,2) circle (0.1); \draw [fill] (4,0) circle (0.1);
\draw  (5,1.0)edge[line width=1.5pt,->](7,1.0);
\begin{scope}[shift={(6,0)}]
\draw [thick,-] (2,2) -- (2,0) (4,0) -- (4,2);
\draw [thick,dashed,-] (4,0) -- (2,0) (2,2) -- (4,2);
\draw [fill] (2,0) circle (0.1); \draw [fill] (2,2) circle (0.1);
\draw [fill] (4,2) circle (0.1); \draw [fill] (4,0) circle (0.1);
\end{scope}
\end{tikzpicture}
\end{center}
It is known that the switch chain converges to the uniform distribution on $\Gnd$, see~\cite{fulkerson,LMM}.
Here we consider a also restricted switch chain, in which a
switch is only allowed if it changes the set of triangles in the graph. This is achieved by \ds es.
A \dsp\ is designed to increase the number of triangles, and a \dsm\ to decrease this number. See Fig.~\ref{fig:triswitch}, where $a_2a_1va_3a_4$ is a path in $G$ such that $a_1a_3$ and $a_2a_4$ are non-edges, which is transformed to a triangle $va_1a_3$ and a vertex-disjoint edge $a_2a_4$ and vice versa.

\begin{figure}[!htb]
\centerline{
\begin{tikzpicture}[scale=0.6]
\tikzset{empty/.style={rectangle,draw=none,fill=none}}
\draw [thick,-] (4,0) -- (2,0) -- (0,1) -- (2,2) -- (4,2);
\draw [dashed,thick] (4,0)--(4,2) (2,0)--(2,2) ;
\draw [fill] (0,1) circle (0.1); \draw [fill] (2,0) circle (0.1);
\draw [fill] (2,2) circle (0.1);
\draw [fill] (4,2) circle (0.1); \draw [fill] (4,0) circle (0.1);
\node [left] at (-0.1,1)  {$v$};
\node [above] at (2,2.1)  {$a_1$}; \node [below] at (2,-0.1)  {$a_3$};
\node [above] at (4,2.1)  {$a_2$}; \node [below] at (4,-0.1)  {$a_4$};
\draw  (6,1.3)edge[line width=1.5pt,->](8,1.3);
\draw (7,1.3)node[empty,label=above:\dsp] {} ;
\draw  (6,0.8)edge[line width=1.5pt,<-](8,0.8);
\draw (7,0.9)node[empty,label=below:\dsm] {} ;
\begin{scope}[shift={(10,0)}]
\draw [thick,-] (0,1) -- (2,2) -- (2,0) -- (0,1) (4,0) -- (4,2);
\draw [thick,dashed,-] (4,0) -- (2,0) (2,2) -- (4,2);
\draw [fill] (0,1) circle (0.1); \draw [fill] (2,0) circle (0.1);
\draw [fill] (2,2) circle (0.1);
\draw [fill] (4,2) circle (0.1); \draw [fill] (4,0) circle (0.1);
\node [left] at (-0.1,1)  {$v$};
\node [above] at (2,2.1)  {$a_1$}; \node [below] at (2,-0.1)  {$a_3$};
\node [above] at (4,2.1)  {$a_2$}; \node [below] at (4,-0.1)  {$a_4$};
\end{scope}
\end{tikzpicture}}
\caption{ The \ds es}\label{fig:triswitch}
\end{figure}
We denote the transitions in the triangle process as \ds es. We may use \dsp\ to denote
a switch which is known to create a triangle, and \dsm\ for a switch which is known
to break a triangle. However, a given \ds\ can simultaneously create one or more triangles
and break one or more triangles, in which case it can be viewed as either a \dsp\ or a \dsm.
We reserve the term ``switch'', without qualification, for the transitions of the switch chain,
as studied in~\cite{CDG07}, for example.
The fact that a single \ds\ can destroy and introduce several triangles at once
significantly complicates the analysis.  Working with a Metropolis chain sidesteps this issue, as it allows us to specify the stationary distribution of the chain. 

Let
\begin{equation}
\label{eq:mu-def}
\mu(\dsq)=M_2^3/(6M^3).
\end{equation} 
We show in Lemma~\ref{prob:lem05} below that the expected number of triangles in a 
uniformly chosen random graph from $\Gnd$ is $\mu(\dsq)$
asymptotically if $d_1=o(\sqrt{M})$. The Metropolis switch chain introduced in Section~\ref{sec:chain} (see Figure~\ref{fig:switch-metropolis-unmodified})
has a non-uniform stationary distribution $\pi_\lambda$, where graph $G$ has probability proportional to
$\lambda^{t(G)}$, for some parameter $\lambda = \lambda(n) \geq 1$, which we call the \emph{activity}.  When $\lambda >1$
the Metropolis switch chain will generate graphs with more triangles than expected in
a uniform random graph with the same degree sequence. In fact, we will show that
$\Expn_{\pi_{\lambda}}[t(G)]\sim\lambda\mu$ in this distribution, if $\lambda$ and $d_1$ are small enough. (See Theorem~\ref{thm:asympt-Poisson}.)

In order to prove rapid mixing we must modify this chain slightly, to decrease the impact of graphs with an extremely high number of triangles, 
as described in Section~\ref{sec:chain} (see Figure~\ref{fig:modified-triangle-switch-code}).  To demonstrate that this modification does not significantly
impact the chain (over a polynomial number of steps), in 
Section~\ref{sec:distribution} we study the asymptotic distribution 
of the number of
triangles under both the uniform distribution on $\cG_{n,\dsq}$ and the stationary
distribution $\pi_\lambda$ of the Metropolis switch chain. 
We prove that (other than in the upper tail) the
distribution in both cases is asymptotically Poisson,  
under the assumption that $M_2 \geq M$. 
(See Section~\ref{s:assumptions} for some discussion about this assumption.)
These results may be of independent interest.

There has been a lot of work on the distribution of triangle counts in the Erd{\H o}s--Renyi model, 
see for example~\cite{AM} and references therein.  
However, prior work on this problem for graphs with given degree
sequences was restricted to the regular case, see~\cite{Gao, ZGW, MWW}.
In particular, the number of triangles in a random $d$-regular graph
is asymptotically Poisson when $d=o(n^{1/5})$, as proved by McKay,
Wormald and Wysocka~\cite{MWW}.  Z.~Gao and Wormald~\cite{ZGW} proved that the number
of triangles in random $d$-regular graphs is asymptotically normal for 
$d=o(n^{2/7})$, and this was improved to $d=O(n^{1/2})$ by P.~Gao~\cite{Gao}.

\subsection{Markov chains}\label{sec:chain}

For $G\in\Gnd$, 
let $\Nb(v)=\Nb_G(v)$ denote the set of neighbours of $v\in [n]$
and recall that $t(G)$ denotes the number of triangles in $G$. 

We now introduce the Metropolis switch chain with parameter $\lambda\geq 1$: the transition matrix of this chain is shown in 
Figure~\ref{fig:switch-metropolis-unmodified}.
The standard switch chain corresponds to choosing $\lambda=1$.

\begin{figure}[ht!]
\begin{center}
\hspace{\dimexpr-\fboxrule-\fboxsep\relax}\fbox{
\begin{minipage}{0.95\textwidth}
\begin{tabbing}
    At the current state $G\in \Gnd$ do\\
    X \= \kill
    \> choose a pair of non-adjacent edges $F=\{a_1a_2,\, a_3a_4\}$ uniformly at random \\
    \>  choose a perfect matching $F'$ of $\{a_1,a_2,a_3,a_4\}$ uniformly at random\\
    \> let $H:= ([n],E')$ where $E':= \big(E(G) \setminus F\big)\cup F'$ \\
    X \= \kill
    \>  if $F'\cap \big(E(G)\setminus F\big) = \es$ then\\
    \> X \= \kill
    \> \> the next state is $H$ with probability $\min\{ 1, \lambda^{t(H)-t(G)}\}$ \\
    \> otherwise\\
    \> \> stay at $G$\\
    \> end if\\
    end
\end{tabbing}
\end{minipage}}
\caption{The transition procedure of the Metropolis switch chain with parameter~$\lambda$ on $\Gnd$}
\label{fig:switch-metropolis-unmodified}
\end{center}
\end{figure}

Let $P_\lambda$ denote the transition matrix of the Metropolis switch chain.
At every step we have $P_\lambda(G,G)\geq 1/3$, since we choose $F'=F$ with probability $1/3$. 
Hence the Metropolis switch chain is aperiodic, and it is known that 
all graphs with a given degree sequence are connected by switches~\cite{fulkerson,LMM}. 
In other words, the Metropolis switch chain is irreducible on $\Gnd$ for all $n$ and $\dsq$.

Define the probability distribution $\pi_\lambda$ by $\pi_\lambda(G) = \lambda^{t(G)}/Z_\lambda(\dsq)$ for
all $G\in \Gnd$, where
\[ Z_\lambda(\dsq) = \sum_{H\in \Gnd} \lambda^{t(H)}\]
is the normalising factor.
Let $a(\dsq)$ be the number of unordered pairs of non-incident distinct edges
in any element of $\Gnd$, given by
\begin{equation}
\label{ad}
a(\dsq) = \binom{M/2}{2} - \frac{M_2}{2}.
\end{equation}
If $P_\lambda(G,H)>0$
and $G\neq H$ then, assuming without loss of generality that $t(H)\leq t(G)$,
we have
\[ \pi_\lambda(G)\, P_\lambda(G,H) = \frac{\lambda^{t(G)}}{Z_\lambda(\dsq)}\, \frac{1}{3 a(\dsq)}\, \lambda^{t(H)-t(G)} =
      \frac{\lambda^{t(H)}}{Z_\lambda(\dsq)}\, \frac{1}{3 a(\dsq)} = \pi_\lambda(H)\, P_\lambda(H,G).\]
Hence the Metropolis switch chain satisfies the detailed balance equations with respect to $\pi_\lambda$,
so $\pi_\lambda$ is the unique stationary distribution of the chain. 

\begin{figure}[ht!]
\begin{center}
\hspace{\dimexpr-\fboxrule-\fboxsep\relax}\fbox{
\begin{minipage}{0.95\textwidth}
\begin{tabbing}
    At the current state $G\in \Gnd$ do\\
    X \= \kill
    \> choose a pair of non-adjacent edges $F=\{a_1a_2,\, a_3a_4\}$ uniformly at random \\
    \>  choose a perfect matching $F'$ of $\{a_1,a_2,a_3,a_4\}$ uniformly at random\\
    \> let $H:= ([n],E')$ where $E':= \big(E(G) \setminus F\big)\cup F'$ \\
    X \= \kill
    \>  if $F'\cap \big(E(G)\setminus F\big) = \es$ then\\
    \> X \= \kill
    \> \> the next state is $H$ with probability $\min\{ 1, \lambda^{\min\{t(H),\nu\}-\min\{t(G),\nu\}}\}$ \\
    \> otherwise\\
    \> \> stay at $G$\\
    \> end if\\
    end
\end{tabbing}
\end{minipage}}
\caption{The transition procedure of the modified Metropolis switch chain with parameter~$\lambda$ on $\Gnd$, where $\nu$ is defined in (\ref{eq:nu-def})}
\label{fig:switch-metropolis}
\end{center}
\end{figure}

For technical reasons, we must also introduce a modified Metropolis switch chain, obtained by 
replacing $t(H)$ by $\min\{t(H),\nu\}$ and replacing $t(G)$ by $\min\{ t(G),\nu\}$ in
the transition procedure Figure~\ref{fig:switch-metropolis-unmodified},
where 
\begin{equation} 
\label{eq:nu-def}
\nu=\nu(n):= \log n/\log \log n.
\end{equation}
Note that this does not change the set of transitions, but ``dampens'' the acceptance probability for states with
a very high number of triangles. (Here and throughout the paper, $\log$ denotes the natural logarithm.)
The transition procedure of the modified Metropolis chain with
parameter $\lambda\geq 1$ is given in Figure~\ref{fig:switch-metropolis}.

Adapting the detailed balance argument given above shows that the modified Metropolis switch chain has stationary distribution $\hat{\pi}_\lambda$ which assigns to graph $G\in\Gnd$ a probability proportional to
$\lambda^{\min\{ t(G),\nu\}}$.  We will see in Section~\ref{sec:distribution} that the
modified Metropolis switch chain is polynomial-time
indistinguishable from the Metropolis switch chain when $\lambda$ and $\mu$ are
small enough: see Remark~\ref{rem:indistinguishable}.

We can also apply the same machinery to the triangle switch chain which,
recall, only performs transitions that change the set of triangles in the graph.
The transition procedure of the modified Metropolis \ds\ chain is shown in 
Figure~\ref{fig:modified-triangle-switch-code}. (As above, the word 
``modified'' indicates the use of $\min\{ t(G),\nu\}$, rather than $t(G)$.)
We prove in Section~\ref{sec:irreducible} below that the modified Metropolis \ds\ chain 
is irreducible, and hence ergodic, provided the minimum degree $d_n\geq 3$.
The stationary distribution of the modified Metropolis \ds\ chain is $\hat{\pi}_\lambda$,
by detailed balance again.
%
%

\begin{figure}[ht!]
\begin{center}
\hspace{\dimexpr-\fboxrule-\fboxsep\relax}\fbox{
\begin{minipage}{0.95\textwidth}
\begin{tabbing}
    At the current state $G\in \Gnd$ do\\
    X \= \kill
    \> choose a pair of non-adjacent edges $F=\{a_1a_2,\, a_3a_4\}$ uniformly at random \\
    \>  choose a perfect matching $F'$ of $\{a_1,a_2,a_3,a_4\}$ uniformly at random\\
    \> let $H:= ([n],E')$ where $E':= \big(E(G) \setminus F\big)\cup F'$ \\
    X \= \kill
    \>  if $F'\cap \big(E(G)\setminus F\big) = \es$ and\\
    \> \phantom{XXXXXX}
     $\big(\Nb(a)\cap \Nb(a')\big)\setminus \{a_1,a_2,a_3,a_4\}\neq \es$ for some $aa'\in F\cup F'$ then\,\, \\
    \> X \= \kill
    \> \> the next state is $H$ with probability $\min\{ 1, \lambda^{\min\{t(H),\nu\}-\min\{t(G),\nu\}}\}$ \\
    \> otherwise\\
    \> \> stay at $G$\\
    \> end if\\
    end
\end{tabbing}
\end{minipage}}
\caption{The transition procedure of the modified Metropolis \ds\ chain with parameter~$\lambda$ on $\Gnd$, where $\nu$ is defined in (\ref{eq:nu-def})}
\label{fig:modified-triangle-switch-code}
\end{center}
\end{figure}

Comparing the transition procedure in Fig.~\ref{fig:modified-triangle-switch-code} with the \ds\ from Fig.~\ref{fig:triswitch}, we see that taking $aa' = a_1a_3$ satisfies the condition
     $\big(\Nb(a)\cap \Nb(a')\big)\setminus \{a_1,a_2,a_3,a_4\}\neq \es$.

\subsection{Our results}\label{sec:our-results}

Our first result shows that if the activity $\lambda$ and the maximum degree
are sufficiently small, then the modified Metropolis switch chain
is rapidly mixing whenever the (standard) switch chain is rapidly mixing.
The proof uses the ``two-stage direct canonical path construction'' method
from~\cite{CDGH}. 

\begin{thm}
\label{thm:Metropolis-switch}
Let $\mathcal{D}$ be a family of graphical sequences such that
for all $\dsq\in\mathcal{D}$ of length~$n$ we have $d_1\geq d_2 \geq \cdots \geq d_n\geq 3$.
Let $\lambda=\lambda(n)\geq 1$ be a function of $n$.
Suppose that there exists an $\alpha\in(0,1)$ such that $\lambda\mu(\dsq)\leq \log^\alpha n$
for every $\dsq\in\mathcal{D}$ of length $n$ and all large enough $n$.
If the switch chain is rapidly mixing on $\Gnd$ for all $\dsq\in\mathcal{D}$ then
the same is true for the modified Metropolis switch chain with parameter~$\lambda$.
\end{thm}

We will use Theorem~\ref{thm:Metropolis-switch} as a convenient stepping-stone in our proof of Theorem~\ref{thm:rapid},
stated below.

Next we turn to the \ds\ chain.
Recall that a Markov chain is irreducible if the underlying graph of the Markov
chain is connected.  This graph has an edge corresponding to each transition
which occurs with positive probability. Hence, if two chains have the same
set of positive-probability transitions and one is irreducible then so is
the other.

We show that the \ds\ chain (or, equivalently, the modified Metropolis \ds\ chain) is irreducible on $\Gnd$. This was proved for regular graphs in~\cite{CDG21}, by showing that (a) every $d$-regular graph can be transformed by a sequence of \ds es into the disjoint union of $\lfloor n/(d+1)\rfloor$ copies of $K_{d+1}$ and one other small component, called a ``fragment'', and (b) 
using \ds es, it is possible to 
transform any set of $\lfloor n/(d+1)\rfloor$ components, each isomorphic to $K_{d+1}$, into any other such set of components,
and to transform fragments into other fragments.
This approach is not suitable for proving rapid mixing of the \ds\ chain by existing methods.
Here we will use the ``two-stage direct canonical path construction'' method
from~\cite{CDGH} to prove rapid mixing, and this informs our approach
to proving irreducibility.

It is known that the set of graphs with a given degree sequence is connected under switches~\cite{fulkerson,LMM}. In other words, the switch chain is irreducible on $\Gnd$ for all $n$ and $\dsq$.
Hence in order to prove that the \ds\ chain is irreducible on $\Gnd$, it suffices
to show that every switch $(G,H)$ can be simulated using a sequence of
\ds es. To do this we will specify a sequence
\[ \sigma_{GH}:\quad  G = X_0, X_1,\ldots, X_{\kappa} = H\]
such that $(X_i,X_{i+1})$ is a \ds\ for $i=0,\ldots, \kappa -1$.
The sequence $\sigma_{GH}$ is called a \emph{simulation path} for $(G,H)$
of length $\kappa$.
Rather than simply prove existence of such a path, we will specify one uniquely
for each possible switch $(G,H)$, and show that all our simulation paths have length
	at most 5.
We will do this for degree sequences with minimum degree at least~3, proving the
following.

\begin{thm} \label{thm:irreducible}
Let $\dsq$ be a graphical degree sequence with $d_1\geq d_2 \cdots \geq d_n\geq 3$
and let $\lambda > 0$ be a constant.
Then the modified Metropolis \ds\ chain with parameter~$\lambda$ is irreducible on $\Gnd$.
\end{thm}
Note that this is not true in general if $d_n<3$, see~\cite{CDG21}.
The above result can be viewed as a \emph{reconfiguration} result on $\Gnd$, generalising that
given in~\cite{CDG21}. Reconfiguration problems have
been a topic of much recent interest in graph theory.
See Nishimura~\cite{nishimura} for an introduction.
Since Theorem~\ref{thm:irreducible} depends only on the set of transitions of the Markov
chain, and not their probabilities, the value of $\lambda$ is immaterial
as long as it is positive.

Our third main result gives a condition under which we can deduce rapid mixing
of the modified \ds\ chain with parameter~$\lambda$ from rapid mixing of the switch chain. Recall the definition of $\mu=\mu(\dsq)$ from (\ref{eq:mu-def}).

\begin{thm} \label{thm:rapid}
Let $\mathcal{D}$ be a family of graphical degree sequences such that
for all $\dsq\in\mathcal{D}$ of length~$n$ we have $d_1\geq d_2 \geq \cdots \geq d_n\geq 3$.
	Let $\lambda=\lambda(n)\geq 1$ be a function of $n$.
	Suppose that there exists an $\alpha\in(0,1)$ such that $\lambda\mu(\dsq)\leq \log^\alpha n$
	for every $\dsq\in\mathcal{D}$ of length $n$ and all large enough $n$.
If the switch chain is rapidly mixing on $\Gnd$ for all $\dsq\in\mathcal{D}$ then
the same is true for the modified \ds\ chain with parameter~$\lambda$.
\end{thm}

The switch chain is known to be rapidly mixing for large classes of graphical sequences, in particular regular sequences~\cite{CDG07,GS} and all graphical sequences from P-stable families~\cite{hungarians}. By Theorem~\ref{thm:rapid}, the \ds\ chain is also rapidly
mixing for these graphical sequences, whenever $\lambda$ and $\dsq$ satisfy the conditions of the theorem.  A special case is stated below.
In particular, this corollary implies that if the maximum degree $d_1$ is bounded above
by a positive constant $C$ for all $\dsq\in\mathcal{D}$
then the \ds\ chain is rapidly mixing for all $\dsq\in\mathcal{D}$.

\begin{cor}\label{cor:very-sparse}
Let $\mathcal{D}$ be a family of graphical sequences such that
there exists some constant $\alpha\in (0,1)$ such that
$ (\log n)^{\alpha/3} \geq d_1\geq d_2 \cdots \geq d_n \geq 3$ 
for all $\dsq\in\mathcal{D}$. Then the modified \ds\ chain with 
parameter $\lambda$ is rapidly mixing for any constant 
$\lambda >1$ and all large enough $n$.
\end{cor}

\begin{proof}
Greenhill and Sfragara~\cite[Theorem~1.1]{GS} proved that the switch chain
is rapidly mixing for any family $\mathcal{D}$ of degree sequences which satisfy 
$d_n\geq 1$ and $3\leq d_1\leq \nfrac{1}{3}\sqrt{M}$.  
These conditions hold under our assumptions when $n$ is large enough. 
Furthermore,
\[ \lambda \mu(\dsq)\leq \dfrac{\lambda}{6}\, d_1^3 \leq \dfrac{\lambda}{6}\, \log^{\alpha} n \leq
\log^{\alpha'}n
  \]
for some slightly larger constant $\alpha'\in (0,1)$, when $n$ is large enough.  
Applying Theorem~\ref{thm:rapid} completes the proof.
\end{proof}

\bigskip

We now describe the structure of the rest of the paper.
In Section~\ref{sec:background-paths} we give some background on the use of canonical paths
for bounding the mixing time of Markov chains and describe 
the two-stage direct canonical path construction method. This machinery is used in 
Section~\ref{sec:mod-switch} to prove Theorem~\ref{thm:Metropolis-switch}.
Then in Section~\ref{sec:irreducible} we prove the irreducibility 
result for the \ds\ chain, Theorem~\ref{thm:irreducible},
and in Section~\ref{sec:rapid} we prove the rapid mixing
result for the \ds\ chain, Theorem~\ref{thm:rapid}. 
Finally in Section~\ref{sec:distribution} we will analyse the distribution of
the number of triangles under both the uniform distribution on $\Gnd$,
and the Gibbs distribution $\pi_{\lambda}$ on $\Gnd$.
In particular, the results of Section~\ref{sec:distribution} imply that
if $M_2\geq M$, the Metropolis switch chain and the modified Metropolis switch chain 
are polynomial-time indistinguisable under the assumptions of Theorem~\ref{thm:Metropolis-switch},
and similarly for the Metropolis \ds\ chain and the modified Metropolis \ds\ chain 
under the assumptions of Theorem~\ref{thm:rapid}.  (The Metropolis \ds\ chain
has transition procedure which is obtained from Figure~\ref{fig:modified-triangle-switch-code} by replacing 
$\min\{ t(H),\nu\}$ with $t(H)$ and replacing $\min\{ t(G),\nu\}$ with $t(G)$.)
See Remark~\ref{rem:indistinguishable}. 

\subsection{Background on canonical paths and mixing time}\label{sec:background-paths}


Let $\cM$ be an ergodic Markov chain with finite state space
$\Omega$, transition matrix $P$ and stationary distribution $\pi$.
Let $\mathcal{G}=(\Omega,E(\cM))$ be the graph underlying $\cM$,
so  the edge set of $\mathcal{G}$
corresponds to (non-loop) transitions of $\cM$.

Given two probability distributions $\sigma$, $\rho$ on $\Omega$,
the \emph{total variation distance} between them, denoted $\dtv(\sigma,\rho)$,
is defined by
\begin{equation}
\label{eq:dtv}
\dtv(\sigma,\rho) = \tfrac{1}{2}\sum_{x\in\Omega} |\sigma(x) - \rho(x)|
  = \max_{A\subseteq \Omega}\, (\sigma(A) - \rho(A)).
\end{equation}
The \emph{mixing time} $\tau(\varepsilon)$ of $\cM$ is given by
\[ \tau(\varepsilon)  = \max_{x\in \Omega} \,\min\,\{t \ge 0\ |\ \dtv(P^t_x,
\pi) \le \varepsilon\},\]
where $P^t_x$ is the distribution of the random state $X_t$ of the
Markov chain after $t$ steps with initial state $x$.

Let the eigenvalues of $P$ be
\[ 1 = \mu_0> \mu_1 \geq \mu_2 \cdots \geq \mu_{N-1} > -1 \]
and define $\mu_\ast = \max\{ \mu_1,\, |\mu_{N-1}|\}$.
Then, see for example~\cite[Proposition 1]{sinclair},
the mixing time of $\cM$ satisfies
\begin{equation}
\label{mixing}
 \tau(\varepsilon) \leq
 (1-\mu_\ast)^{-1}
 \left(\log(1/\pi^\ast) + \log(\varepsilon^{-1}) \right)
\end{equation}
where $\pi^\ast = \min_{x\in\Omega}\pi(x)$ is the smallest stationary probability.
In applications we often focus on the second-largest eigenvalue $\mu_1$,
if necessary by making the Markov chain lazy (replacing the transition matrix $P$
	by $(I+P)/2$ to ensure that all eigenvalues are nonnegative).
Alternatively, it follows from a result of Diaconis and Saloff-Coste~\cite[p.~702]{DSC} that the smallest eigenvalue of $\cM$ satisfies
\begin{equation}
\label{eq:smallest}
   (1 + \mu_{N-1})^{-1} \leq  \max_{x\in\Omega} \frac{1}{2\, P(x,x)}.
   \end{equation}

Jerrum and Sinclair~\cite{jerrum89conductancepermanent} introduced the
\emph{canonical path} method for bounding $\mu_1$.
For each pair $x,y\in\Omega$, we choose a path $\gamma_{xy}$
from $x$ to $y$ in the graph $\mathcal{G}$ underlying the Markov chain.
A \emph{canonical} choice of paths seeks to produce a small value of
the critical parameter $\bar{\rho}(\Gamma)$, the \emph{congestion} of the chosen set of paths
$\Gamma = \{ \gamma_{xy} \mid x,y\in\Omega \}$, defined by
\[ \bar{\rho}(\Gamma) = \max_{e\in E(\cM)}  Q(e)^{-1} \sum_{\substack{x,y\in\Omega\\ e\in\gamma_{xy}}}\, \pi(x)\pi(y)|\gamma_{xy}|
\]
where $Q(e) = \pi(u)\, P(u,v)$ when $e=uv\in E(\cM)$.
Then
\begin{equation}
 (1-\mu_1)^{-1}\leq \bar{\rho}(\Gamma),
\label{canonicalpath}
\end{equation}
so a low congestion can lead to a good bound on the mixing time, using (\ref{mixing}).
See~\cite[Theorem 5]{sinclair} or~\cite{jerrum03lectures}.

In~\cite{CDGH}, the \emph{two-stage direct canonical path construction method} was
described.  Here we have two ergodic Markov chains $\cM$ and $\cM'$,
with state spaces $\Omega$ and $\Omega'$, respectively.  We assume that we
have a set of canonical paths for $\cM'$. Furthermore, for each
each transition $(Y,Z)$ of $\cM'$, we specify a (canonical)
sequence $\sigma_{YZ}$ of transitions of $\cM$ to simulate $(Y,Z)$:
\[ \sigma_{YZ}:\quad  Y = X_0,X_1,\ldots, X_\kappa = Z\]
such that $(X_i,X_{i+1})$ is a transition of $\cM$.
We call $\sigma_{YZ}$ a $(\cM,\cM')$-\emph{simulation path} and
let $\Sigma = \{ \sigma_{YZ}\mid (Y,Z)\in E(\cM')\}$ be the set
of all specified $(\cM,\cM')$-simulation paths (one for each transition
of $\cM'$).

For each transition $e\in E(\cM)$, let
\[ \Sigma(e) = \{ \sigma\in \Sigma \mid e\in\sigma\}\]
be the set of all specified simulations paths which use the transition $e$.
We define the parameters
\begin{equation}
\label{eq:parameters}
   \ell(\Sigma) = \max_{\sigma\in \Sigma} |\sigma |, \quad
    B(\Sigma) = \max_{e\in E(\cM)} |\Sigma(e)|
 \end{equation}
representing the maximum length of any simulation path, 
and the maximum
number of simulation paths through a given transition of $\cM$,
respectively.  Finally, the  \emph{simulation gap} between the two chains is defined by
\begin{equation}
\label{eq:simulationgap}
D(\cM,\cM') =
  \max_{\substack{uv\in E(\cM)\\zw\in E(\cM')}}
     \frac{\pi'(z) P'(z,w)}{\pi(u) P(u,v)}
\end{equation}

and the \emph{stationary ratio} is defined by
\begin{equation}
\label{eq:stationaryratio}
R(\cM,\cM') =
  \max_{x,y\in \Omega}
     \frac{\pi(x)\, \pi(y)}{\pi'(x) \pi'(y)}= \left(\max_{x\in \Omega}\frac{\pi(x)}{\pi'(x)}\right)^2.
\end{equation}

Note that $\pi$ and $\pi'$ are both positive on their domains as
$\cM$, $\cM'$ are ergodic and reversible.

The following is a slight simplification of~\cite[Theorem~2.1]{CDGH},
tailored to the case that $\Omega=\Omega'$.
(If $\Omega=\Omega'$ then the surjection $h$ mentioned
in~\cite[Theorem~2.1]{CDGH} can be taken to be the identity.)
We will also take the opportunity to correct a (fortunately) minor
error in~\cite[Theorem~2.1]{CDGH}, where the factor $R(\cM,\cM')$ was omitted.

\begin{thm} 
\label{thm:CDGH}
Let $\cM$ and $\cM'$ be two ergodic Markov chains on a
set $\Omega$. Let $\Gamma'$ be a set of canonical paths in $\cM'$
and let $\Sigma$ be a set of $(\cM,\cM')$-simulation paths.
Then there exists a set of canonical paths in $\cM$ whose congestion
satisfies
\[ \bar{\rho}(\Gamma)\leq D(\cM,\cM')\, R(\cM,\cM')\, \ell(\Sigma)\,
         B(\Sigma)\, \bar{\rho}(\Gamma').
	 \]
\end{thm}

\begin{proof}
In the proof of~\cite[Theorem~2.1]{CDGH} it is shown that
\[
\bar{\rho}(\Gamma)
\leq D(\cM,\cM')\, \ell(\Sigma)\, \max_{uv\in E(\cM)} \,
	\sum_{\substack{zw\in E(\cM')\\\sigma_{zw}\in \Sigma(uv)}}
		\frac1{\pi'(z)P'(z,w)}
	\sum_{\substack{x,y\in\Omega'\\zw\in\gamma'_{xy}}}
		\pi(x)\pi(y)\,|\gamma'_{xy}|
		\]
and that
\[
\max_{uv\in E(\cM)} \,
	\sum_{\substack{zw\in E(\cM')\\\sigma_{zw}\in \Sigma(uv)}}
		\frac1{\pi'(z)P'(z,w)}
	\sum_{\substack{x,y\in\Omega'\\zw\in\gamma'_{xy}}}
		\pi'(x)\pi'(y)\,|\gamma'_{xy}|
	\leq B(\Sigma)\, \bar{\rho}(\Gamma').
	\]
	The result follows as $\pi(x)\pi(y)\leq R(\cM,\cM')\, \pi'(x)\pi'(y)$.
\end{proof}

We remark that in the analysis of the flip chain given in~\cite{CDGH},
the above theorem is applied twice.  In one application the missing factor
$R$ is exactly~1, and in the other application it can be bounded above
by a constant close to 1 (say~2).  So the stated bound on the mixing time of the flip
chain in~\cite{CDGH} is too small by a factor of at most~2, which is of no consequence.
However, in the proof of Theorem~\ref{thm:Metropolis-switch} the ratio $R(\cM,\cM')$ becomes rather more 
significant as we compare the uniform distribution to a highly non-uniform 
distribution.

\subsection{Rapid mixing of the modified Metropolis switch chain}\label{sec:mod-switch}

We can immediately prove Theorem~\ref{thm:Metropolis-switch} using the two-stage
direct canonical path construction method.  As well as being of independent interest, this result 
forms a useful stepping-stone in our proof of Theorem~\ref{thm:rapid}.

\begin{proof}[Proof of Theorem~\ref{thm:Metropolis-switch}]\
Let $\cM_\lambda$ denote the modified Metropolis switch chain with parameter $\lambda$, as given in Figure~\ref{fig:switch-metropolis}, and let $\cM'$ denote the standard switch chain with uniform stationary
distribution: that is, $\cM'$ has transition procedure given by Figure~\ref{fig:switch-metropolis} with $\lambda=1$.  
Suppose that $\cM'$ is rapidly mixing on all degree
sequences in some family $\mathcal{D}$.  That is, there exists a polynomial
$p(n)$ such that if $\dsq\in\mathcal{D}$ and $\dsq$ has length $n$ then
the mixing time of the switch chain on $\Gnd$ is at most $p(n)$.

Using results of Sinclair~\cite{sinclair}, Guruswami~\cite[Theorem~4.9]{guruswami} 
proved that there exists a set $\Gamma'$  of canonical paths for $\cM'$
and a polynomial $q(n)$ such that $\bar{\rho}(\Gamma')\leq q(n)$ for
all $\dsq\in\mathcal{D}$ with length $n$.
Since the switch chain and the modified Metropolis switch chain have the
same set of allowable transitions, we may take each simulation path to have
length~1.  Let $\Sigma$ be the set of these simulations paths.  Then
$\ell(\Sigma)=1$ and $B(\Sigma)=1$.

The normalising factor $\hat{Z}_\lambda(\dsq)$ 
for the stationary distribution of the modified Metropolis switch chain 
satisfies
\begin{equation}
\label{eq:Z-upper}
\hat{Z}_\lambda(\dsq) = \sum_{G\in\Gnd} \lambda^{\min\{t(G),\, \nu\}}
    \leq \lambda^{\nu}\, |\Gnd|.
\end{equation}
The switch chain $\cM'$ has uniform stationary distribution over $\Gnd$
and transition probability $P'(G,H) = \frac{1}{3 a(\dsq)}$ whenever $G\neq H$
differ by a switch. 
Hence the simulation gap $D(\cM,\cM')$ satisfies
\[ D(\cM,\cM') = \max_{uv\in E(\cM)} \frac{1}{3 a(\dsq)\, |\Gnd|}\cdot \frac{3 a(\dsq)\, \hat{Z}_{\lambda}(\dsq)}{\min\{\lambda^{\min\{t(u),\nu\}},\lambda^{\min\{t(v),\nu\}}\}} \leq \lambda^{\nu}.
\]
(Since $\lambda \geq 1$, the maximum over $uv\in E(\cM)$ is obtained when
$t(u)=0$ or $t(v)=0$.)
The stationary ratio satisfies
\[ R(\cM,\cM') = \max_{x\in \Gnd} \left(\frac{|\Gnd|\, \lambda^{\min\{t(x),\nu\}}}{\hat{Z}_\lambda(\dsq)}\right)^2 \leq \left(\max_{x\in\Gnd} \lambda^{\min\{t(x),\nu\}}\right)^2 \leq \lambda^{2\nu}.
\]
Since we take $\nu = \frac{\log n}{\log \log n}$, under the assumptions of the
theorem we have 
\begin{equation}
\label{eq:upper-lambda-nu}
\lambda^{\nu} \leq \big( \log^\alpha \big)^{\log n/\log \log n} = e^{\alpha \log n} = n^\alpha.
\end{equation}
Hence, by Theorem~\ref{thm:CDGH}, there exists a set $\Gamma$ set of canonical paths for
$\cM$ (in fact $\Gamma=\Gamma'$) with congestion
\[ \bar{\rho}(\Gamma)\leq n^{3\alpha}\, \bar{\rho}(\Gamma'). \]
It follows that
\[ (1-\mu_1)^{-1}\leq \bar{\rho}(\Gamma) \leq n^{3\alpha} \, q(n).\]
Next, the modified Metropolis switch chain has self-loop probability of
at least $1/3$ on every state, so (\ref{eq:smallest}) implies that
\[ (1+\mu_{N-1})^{-1} \leq \dfrac{3}{2}.\]
Combining the two previous inequalities implies that
\[ (1 - \mu_*)^{-1}\leq n^{3\alpha}\, q(n).\]
To apply (\ref{mixing}) we note that under the assumptions of the theorem,
the minimum stationary probability for $\cM$ satisfies
\[ \log(1/\pi^*) = \log\big(\hat{Z}_\lambda(\dsq)\big) \leq \log\big(\lambda^{\nu}\big) + \log\big(|\Gnd|\big) \leq \log n + M\log M\]
using (\ref{eq:Z-upper}).
This uses the fact, which can be proved 
using the configuration model~\cite{bollobas}, that
\begin{equation}
\label{eq:log-size-state-space}
\log |\Gnd| \leq M\log M
\end{equation}
where $M=M(\dsq) = \sum_{j\in [n]} d_j$. 
Hence using (\ref{mixing}) we conclude that the modified Metropolis switch chain
is rapidly mixing, with mixing time bounded above by 
\[ \tau(\varepsilon) \leq n^{3\alpha} \, q(n)\, \big( M\log M + \log n + \log(\varepsilon^{-1})\big). \]
This completes the proof of Theorem~\ref{thm:Metropolis-switch}.
\end{proof}

We can see from this proof why we dampen the impact of the
upper tail of the distribution of $t(G)$ by replacing $t(G)$ with $\min\{ t(G),\nu\}$
in the transition procedure.  The factor $D(\cM,\cM')$ 
involves the ratio $\hat{Z}_{\lambda}(\dsq)/|\Gnd|$, which is the
expected value of $\lambda^{\min\{ t(G),\nu\}}$ under the uniform distribution on
$\Gnd$.  Without modification, this factor would become the expected value of
$\lambda^{t(G)}$ with respect to the uniform distribution on $\Gnd$, and we
do not believe that this quantity is polynomially bounded.

\section{Irreducibility}\label{sec:irreducible}

In this section we prove that every switch can be simulated by a sequence of at most five \ds es. 

\begin{thm} \label{thm:five}
Let $(G,H)$ be a switch, where $G,H\in\Gnd$. There is a \ds\ simulation path of length at most~5 which simulates the switch $(G,H)$.
\end{thm}

We will present the constructions diagrammatically, with the following conventions: solid lines represent known edges of $G$, dotted lines represent known non-edges, and dashed lines represent non-edges involved in the current \ds.

Suppose that a \ds\ $(G,H)$ replaces the edges $a_1a_2$,\, $a_3a_4 $ with the
edges $a_1a_3$,\, $a_2a_4 $. We classify switches into three types,
depending on whether one or
both of the \emph{diagonals} $a_1a_4$, $a_2a_3$ is present or absent, as shown
in Fig.~\ref{fig:3cases}.  Note that the diagonals
have the form $a_ja_{5-j}$ for $j=1,2$.

\begin{figure}[ht!]
\begin{center}
\begin{tikzpicture}[scale=0.74,inner sep=0pt]
\node (a1) [b,label=above left:$a_1$] at (0,2) {};
\node (a2) [b,label=above right:$a_2$] at (2,2) {} ;
\node (a3) [b,label=below left:$a_3$] at (0,0)  {};
\node (a4) [b,label=below right:$a_4$] at (2,0)  {};
\draw  (a1) -- (a2) (a3) -- (a4);
\draw [dashed] (a1) -- (a3) (a2) -- (a4);
\draw [dotted] (a1)-- (a4) (a2)--(a3) ;
\node at (1,-1) {Type (a)};
\begin{scope}[shift={(6,0)}]
\node (a1) [b,label=above left:$a_1$] at (0,2) {};
\node (a2) [b,label=above right:$a_2$] at (2,2) {} ;
\node (a3) [b,label=below left:$a_3$] at (0,0)  {};
\node (a4) [b,label=below right:$a_4$] at (2,0)  {};
\draw  (a1) -- (a2) (a2) -- (a3) (a3) -- (a4) (a1) -- (a4);
\draw [dashed] (a1) -- (a3) (a2) -- (a4) ;
\node at (1,-1) {Type (b)};
\end{scope}
\begin{scope}[shift={(12,0)}]
\node (a1) [b,label=above left:$a_1$] at (0,2) {};
\node (a2) [b,label=above right:$a_2$] at (2,2) {} ;
\node (a3) [b,label=below left:$a_3$] at (0,0)  {};
\node (a4) [b,label=below right:$a_4$] at (2,0)  {};
\draw  (a1) -- (a2) (a3) -- (a4) (a1) -- (a4) ;
\draw [dashed] (a1) -- (a3) (a2) -- (a4);
\draw [dotted]  (a2) -- (a3);
\node at (1,-1) {Type (c)};
\end{scope}
\end{tikzpicture}
\end{center}
\caption{The three types of switch, depending on how many diagonals are present}
\label{fig:3cases}
\end{figure}

Let $D=\{a_1,a_2,a_3,a_4\}$ be the vertices of the switch, and $A_i=\Nb(a_i)\sm D$ 
be the neighbours of $a_i$ outside of $D$, for $i=1,\ldots, 4$. Note that $|A_i|=\deg(a_i)-\deg_D(a_i)$, so $|A_i|\geq d_n-1$ or $|A_i|\geq d_n-2$. Thus, since $d_n\geq3$, we have $A_i\ne\es$ for all $i\in[4]$ in any switch.

\bigskip

\begin{lem}
\label{lem:pivot}
Suppose that $(G,H)$ is a switch which removes the edges $a_1a_2$, $a_3a_4$
and replaces them with the edges $a_1a_3$, $a_2a_4$.
Further suppose that at least one of the following holds:
\begin{itemize}
\item[\emph{(i)}] The switch is Type~\emph{(b)} or Type~\emph{(c)}. That is, at least one diagonal is present.
\item[\emph{(ii)}] The switch is Type~\emph{(a)}, and some pair of distinct elements of
$D = \{ a_1,a_2,a_3,a_4\}$ have a common neighbour in $[n]\setminus D$.
\item[\emph{(iii)}] There is a triangle on vertices $a_j, u, w$ for some vertices $u,w$
which satisfy $u\not\in A_{5-j}$ and $\{u,w\}\cap D = \es$,
where $j\in [4]$.
\end{itemize}
Then there is a \ds\ simulation path of length at most~4 which simulates
the switch $(G,H)$.
\end{lem}

\begin{proof}
For future reference, we will number the cases which arise when defining the
simulation paths. This will be needed later in the proof of rapid mixing.
We treat these cases in order, so we only move to~(II) if~(I) does not
hold, and so on.
\begin{itemize}
\item[(I)]
If the switch $(G,H)$ is itself a \ds\ then the simulation path is $G,H$ which
has length~1.  This occurs
if and only if at least one of $A_1\cap A_2$, $A_1\cap A_3$, $A_3\cap A_4$, $A_2\cap A_4$
	is nonempty.
\item[(II)]
Now suppose that the switch has Type~(a) (both diagonals absent)
and $(A_1\cap A_4)\cup (A_2\cap A_3)$ is non-empty.
After relabelling if necessary, assume that
$A_1\cap A_4\neq\es$.
Then there is a simulation path of length~2, shown below.
\begin{center}
\begin{tikzpicture}[scale=0.74,inner sep=0pt]
\node (v) [b,label=above:${}_{}v$] at (1,2.75) {};
\node (a1) [b,label=above left:$a_1$] at (0,2) {};
\node (a2) [b,label=below right:$a_2$] at (2,0) {} ;
\node (a3) [b,label=below left:$a_3$] at (0,0)  {};
\node (a4) [b,label=above right:$a_4$] at (2,2)  {};
\draw  (a1) -- (a2) (a3) -- (a4) (a1) -- (v) -- (a4);
\draw [dashed] (a1)-- (a4) (a2)--(a3) ;
\draw[dotted] (a1) -- (a3) (a2) -- (a4) ;
\node at (4,1.25) {\Dp} ; \node at (4,0.75) {$\longrightarrow$} ;
\begin{scope}[shift={(6,0)}]
\node (v) [b,label=above:${}_{}v$] at (1,2.75) {};
\node (a1) [b,label=above left:$a_1$] at (0,2) {};
\node (a2) [b,label=below right:$a_2$] at (2,0) {} ;
\node (a3) [b,label=below left:$a_3$] at (0,0)  {};
\node (a4) [b,label=above right:$a_4$] at (2,2)  {};
\draw  (a1) -- (a4) (a2) -- (a3) (a1) -- (v) -- (a4);
\draw [dashed] (a1) -- (a3) (a2) -- (a4) ;
\draw[dotted] (a1) -- (a2) (a3) -- (a4) ;
\node at (4,1.25) {\Dm} ; \node at (4,0.75) {$\longrightarrow$} ;
\end{scope}
\begin{scope}[shift={(12,0)}]
\node (v) [b,label=above:${}_{}v$] at (1,2.75) {};
\node (a1) [b,label=above left:$a_1$] at (0,2) {};
\node (a2) [b,label=below right:$a_2$] at (2,0) {} ;
\node (a3) [b,label=below left:$a_3$] at (0,0)  {};
\node (a4) [b,label=above right:$a_4$] at (2,2)  {};
\draw  (a1) -- (a3) (a2) -- (a4) (a1) -- (v) -- (a4);
\draw [dotted] (a1) -- (a4) (a1) -- (a2) (a3)-- (a4) (a3)--(a2) ;
\end{scope}
\end{tikzpicture}
\end{center}
\item[(III)]
Next, suppose that the switch has Type~(b) (both diagonals present) and
$(A_1\cap A_4)\cup(A_2\cap A_3)$ is non-empty.
After relabelling if necessary, $A_1\cap A_4\neq\es$.
Again, there is a simulation path of length~2, as shown below.
\begin{center}
\begin{tikzpicture}[scale=0.74,inner sep=0pt]
\node (v) [b,label=above:${}_{}v$] at (1,2.75) {};
\node (a1) [b,label=above left:$a_1$] at (0,2) {};
\node (a2) [b,label=below right:$a_2$] at (2,0) {} ;
\node (a3) [b,label=below left:$a_3$] at (0,0)  {};
\node (a4) [b,label=above right:$a_4$] at (2,2)  {};
\draw  (a1) -- (a2) (a3) -- (a4) (a1) -- (v) -- (a4) (a1)-- (a4) (a2)--(a3);
\draw [dashed] (a1) -- (a3) (a2) -- (a4) ;
\node at (4,1.25) {\Dm} ; \node at (4,0.75) {$\longrightarrow$} ;
\begin{scope}[shift={(6,0)}]
\node (v) [b,label=above:${}_{}v$] at (1,2.75) {};
\node (a1) [b,label=above left:$a_1$] at (0,2) {};
\node (a2) [b,label=below right:$a_2$] at (2,0) {} ;
\node (a3) [b,label=below left:$a_3$] at (0,0)  {};
\node (a4) [b,label=above right:$a_4$] at (2,2)  {};
\draw  (a1) -- (v) -- (a4) (a1) -- (a3) (a2) -- (a4) (a1)-- (a2) (a3)--(a4);
\draw [dashed] (a1) -- (a4) (a2) -- (a3);
\node at (4,1.25) {\Dp} ; \node at (4,0.75) {$\longrightarrow$} ;
\end{scope}
\begin{scope}[shift={(12,0)}]
\node (v) [b,label=above:${}_{}v$] at (1,2.75) {};
\node (a1) [b,label=above left:$a_1$] at (0,2) {};
\node (a2) [b,label=below right:$a_2$] at (2,0) {} ;
\node (a3) [b,label=below left:$a_3$] at (0,0)  {};
\node (a4) [b,label=above right:$a_4$] at (2,2)  {};
\draw  (a1) -- (a3) (a1) -- (a4) (a2) -- (a4) (a3)--(a2) (a1) -- (v) -- (a4);
\draw [dotted] (a1) -- (a2) (a3)-- (a4) ;
\end{scope}
\end{tikzpicture}
\end{center}
\item[(IV)]

Note that if none of (I), (II) or (III) hold then $A_i\cap A_j = \es$ for all distinct
$i,j\in [4]$.
Next we suppose that, after relabelling if necessary,
\begin{itemize}
\item at least one of the diagonal edges $a_ja_{5-j}$ is present, and the
symmetric difference $A_i \oplus A_{5-i}$ is nonempty,
where $\{i,j\} = \{1,2\}$, or
\item there is a triangle on vertices $a_j, u, w$ for two distinct
vertices $u,w \in [n]\setminus D$.
(Here we do not care how many diagonal edges are present.)
\end{itemize}
In the first subcase, choose $(i,u)$ to be lexicographically-least.  Then after
relabelling if necessary, we can assume that
$u$ is the lexicographically-least element of $A_1\setminus A_4$. In the
second subcase, choose $(j,u,w)$ to be the lexicographically-least triple satisfying
the requirements.  Then after relabelling if necessary, we can assume that
$j=1$.  In both cases, we have a simulation path of length~2, consisting of
first replacing edges $a_1 u, a_3a_4$ by $a_1a_3, ua_4$ and then replacing
edges $ua_4, a_1a_3$ by edges $a_1 u, a_2a_4$. Under the given conditions,
each of these operations is a \ds, as required.  We illustrate the simulation path
below for the first subcase when exactly one diagonal $a_2a_3$ is present (noting that
it this still gives a sequence of two \ds es if $a_1a_4$ is also present), and in the second subcase when neither diagonal
is present (noting that this still gives a sequence of two \ds es if one or both diagonals are present).

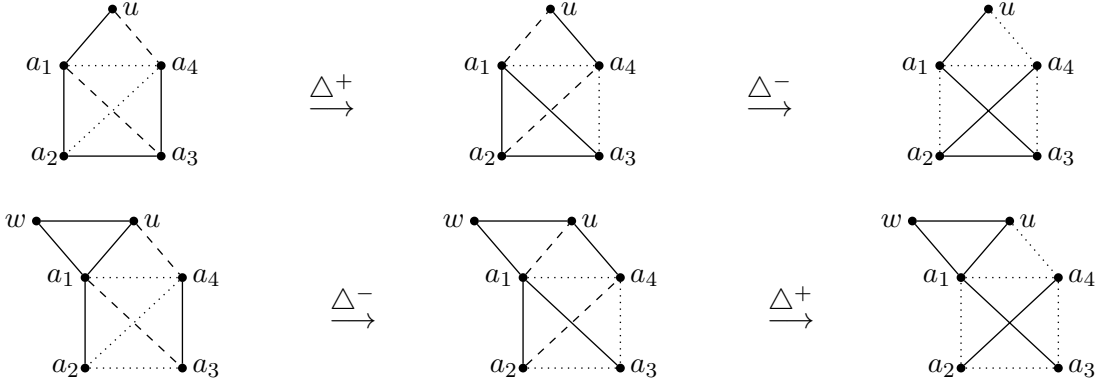
\begin{center}
\begin{tikzpicture}[xscale=0.64,yscale=0.6,inner sep=0pt,font=\small]
\node (u) [b,label=right:\,$u$] at (1,5.25) {};
\node (a1) [b,label=left:$a_1$\,] at (0,4) {};
\node (a2) [b,label=left: $\,a_2$] at (0,2) {};
\node (a3) [b,label=right:$\,a_3$] at (2,2) {} ;
\node (a4) [b,label=right:$\,a_4$] at (2,4)  {};
\draw  (u) -- (a1)-- (a2) -- (a3) -- (a4) ;
\draw [dashed] (u) -- (a4) (a1) -- (a3) ;
\draw[dotted] (a1)--(a4)--(a2) ;
\node at (5.5,3.5) {\Dp} ; \node at (5.5,3) {$\longrightarrow$} ;
\begin{scope}[shift={(9,0)}]
\node (u) [b,label=right:\,$u$] at (1,5.25) {};
\node (a1) [b,label=left:$a_1$\,] at (0,4) {};
\node (a2) [b,label=left: $\,a_2$] at (0,2) {};
\node (a3) [b,label=right:$\,a_3$] at (2,2) {} ;
\node (a4) [b,label=right:$\,a_4$] at (2,4)  {};
\draw  (u) -- (a4) (a1) -- (a2) -- (a3) -- (a1);
\draw [dashed] (a2) -- (a4) (a1) -- (u) ;
\draw[dotted] (a1)--(a4)--(a3) ;
\node at (5.5,3.5) {\Dm} ; \node at (5.5,3) {$\longrightarrow$} ;
\end{scope}
\begin{scope}[shift={(18,0)}]
\node (u) [b,label=right:\,$u$] at (1,5.25) {};
\node (a1) [b,label=left:$a_1$\,] at (0,4) {};
\node (a2) [b,label=left: $\,a_2$] at (0,2) {};
\node (a3) [b,label=right:$\,a_3$] at (2,2) {} ;
\node (a4) [b,label=right:$\,a_4$] at (2,4)  {};
\draw  (u) -- (a1) -- (a3) -- (a2) -- (a4);
\draw[dotted] (u)--(a4)--(a3) (a2)--(a1)--(a4);
\end{scope}
\end{tikzpicture}
\end{center}

\begin{center}
\begin{tikzpicture}[xscale=0.64,yscale=0.6,inner sep=0pt,font=\small]
\node (w) [b,label=left:$w$\,] at (-1,5.25) {};
\node (u) [b,label=right:\,$u$] at (1,5.25) {};
\node (a1) [b,label=left:$a_1$\,] at (0,4) {};
\node (a2) [b,label=left: $\,a_2$] at (0,2) {};
\node (a3) [b,label=right:$\,a_3$] at (2,2) {} ;
\node (a4) [b,label=right:$\,a_4$] at (2,4)  {};
\draw  (a1) -- (w) -- (u) -- (a1)-- (a2)  (a3) -- (a4) ;
\draw [dashed] (u) -- (a4) (a1) -- (a3) ;
\draw[dotted] (a1)--(a4)--(a2) -- (a3) ;
\node at (5.5,3.5) {\Dm} ; \node at (5.5,3) {$\longrightarrow$} ;
\begin{scope}[shift={(9,0)}]
\node (w) [b,label=left:$w$\,] at (-1,5.25) {};
\node (u) [b,label=right:\,$u$] at (1,5.25) {};
\node (a1) [b,label=left:$a_1$\,] at (0,4) {};
\node (a2) [b,label=left: $\,a_2$] at (0,2) {};
\node (a3) [b,label=right:$\,a_3$] at (2,2) {} ;
\node (a4) [b,label=right:$\,a_4$] at (2,4)  {};
\draw  (u) -- (a4) (a1) -- (a2)  (a3) -- (a1) -- (w) -- (u);
\draw [dashed] (a2) -- (a4) (a1) -- (u) ;
\draw[dotted] (a1)--(a4)--(a3) -- (a2) ;
\node at (5.5,3.5) {\Dp} ; \node at (5.5,3) {$\longrightarrow$} ;
\end{scope}
\begin{scope}[shift={(18,0)}]
\node (w) [b,label=left:$w$\,] at (-1,5.25) {};
\node (u) [b,label=right:\,$u$] at (1,5.25) {};
\node (a1) [b,label=left:$a_1$\,] at (0,4) {};
\node (a2) [b,label=left: $\,a_2$] at (0,2) {};
\node (a3) [b,label=right:$\,a_3$] at (2,2) {} ;
\node (a4) [b,label=right:$\,a_4$] at (2,4)  {};
\draw (a1) -- (w) -- (u) -- (a1) -- (a3)  (a2) -- (a4);
\draw[dotted] (u)--(a4)--(a3) (a3) -- (a2)--(a1)--(a4);
\end{scope}
\end{tikzpicture}
\end{center}

\item[(V)]
Next, suppose that the switch has Type~(c), so exactly one diagonal $a_ja_{5-j}$
is present, and that $A_i = A_{5-i}$, and that $G[A_i]$ has at least
one edge, where $\{i,j\} = \{ 1,2\}$.
By relabelling if necessary, we may assume that diagonal $a_1a_4$ is present,
and for ease of notation let $A= A_2 = A_3$.
Let $(u,v)$ be the lexicographically-least pair
of vertices in $A$ such that $uv$ is an edge of $G$.
Then $a_2,u,v$ and $a_3,u,v$ both determine triangles,  while $u\not\in A_1$ and
$v\not\in A_4$ as we are not in (I).   Then there is a simulation path
of length~4 as shown below.

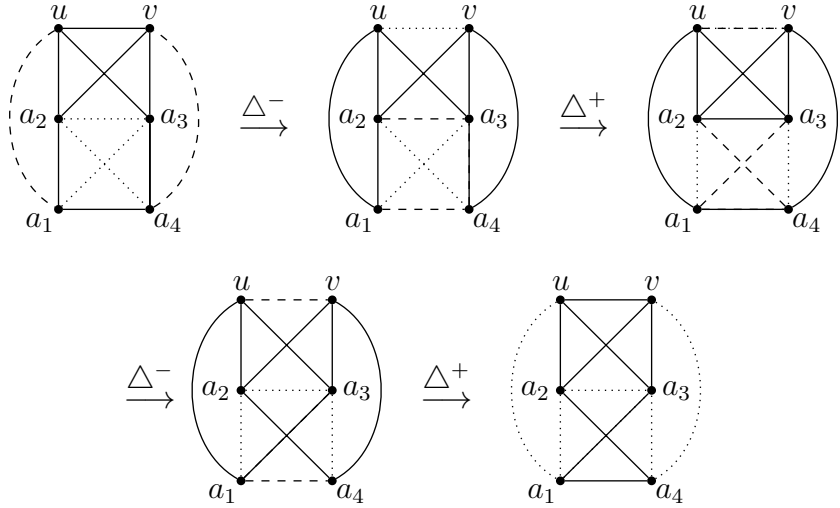
\begin{center}
\begin{tikzpicture}[scale=0.60,inner sep=0pt]
\begin{scope}[shift={(0,0)}]
\node (v) [b,label=above:${}_{}v$] at (2,4) {};
\node (a2) [b,label=left:$a_2$\,] at (0,2) {};
\node (a1) [b,label=below left:$a_1$] at (0,0) {} ;
\node (a4) [b,label=below right:$a_4$] at (2,0)  {};
\node (a3) [b,label=right:\,$a_3$] at (2,2)  {};
\node (u) [b,label=above:${}_{}u$] at (0,4) {};
\draw  (a2) -- (a1) (a1)--(a4)(a2) -- (v) -- (a3) (a3)-- (a4) (v) -- (u) -- (a2) (u)--(a3) (a4) -- (a3) ;
\draw [dashed] (u)edge[bend right=60](a1) (a4)edge[bend right=60](v);
\draw[dotted] (a1)--(a3)--(a2)--(a4) ;
\node at (4.5,2.25) {\small \Dm} ; \node at (4.5,1.75) {$\longrightarrow$} ;
\end{scope}
\begin{scope}[shift={(7,0)}]
\node (v) [b,label=above:${}_{}v$] at (2,4) {};
\node (a2) [b,label=left:$a_2$\,] at (0,2) {};
\node (a1) [b,label=below left:$a_1$] at (0,0) {} ;
\node (a4) [b,label=below right:$a_4$] at (2,0)  {};
\node (a3) [b,label=right:\,$a_3$] at (2,2)  {};
\node (u) [b,label=above:${}_{}u$] at (0,4) {};
\draw  (a2) -- (a1)(a2) -- (v) -- (a3) (a3)-- (a4) (u) -- (a2) (u)--(a3);
\draw [dashed] (a4) -- (a3) (a2) -- (a3) (a1)--(a4);
\draw  (u)edge[bend right=60](a1) (a4)edge[bend right=60](v);
\draw[dotted] (a1)--(a3)(a2)--(a4) (u)--(v);
\node at (4.5,2.25) {\Dp} ; \node at (4.5,1.75) {$\longrightarrow$} ;
\end{scope}
\begin{scope}[shift={(14,0)}]
\node (v) [b,label=above:${}_{}v$] at (2,4) {};
\node (a2) [b,label=left:$a_2$\,] at (0,2) {};
\node (a1) [b,label=below left:$a_1$] at (0,0) {} ;
\node (a4) [b,label=below right:$a_4$] at (2,0)  {};
\node (a3) [b,label=right:\,$a_3$] at (2,2)  {};
\node (u) [b,label=above:${}_{}u$] at (0,4) {};
\draw (a2) -- (v) -- (a3) (a1) -- (a4) (u) -- (a2) (a2) -- (a3) (u)--(a3);
\draw [dashed] (a2) -- (a4) (a1) -- (a3) (u) --(v) (a1)--(a4);
\draw  (u)edge[bend right=60](a1) (a4)edge[bend right=60](v);
\draw[dotted] (a1)--(a2)(a3)--(a4) (u)--(v);
\end{scope}
\begin{scope}[shift={(4,-6)}]
\node at (-2,2.25) {\Dm} ; \node at (-2,1.75) {$\longrightarrow$} ;
\node (v) [b,label=above:${}_{}v$] at (2,4) {};
\node (a2) [b,label=left:$a_2$\,] at (0,2) {};
\node (a1) [b,label=below left:$a_1$] at (0,0) {} ;
\node (a4) [b,label=below right:$a_4$] at (2,0)  {};
\node (a3) [b,label=right:\,$a_3$] at (2,2)  {};
\node (u) [b,label=above:${}_{}u$] at (0,4) {};
\draw (a2) -- (v) -- (a3) (a1) -- (a3) (u) -- (a2) (a2) -- (a4) (u)--(a3) ;
\draw [dashed] (a1) -- (a3) (u) --(v) (a1)--(a4);
\draw  (u)edge[bend right=60](a1) (a4)edge[bend right=60](v);
\draw[dotted] (a1)--(a2)--(a3)--(a4);
\end{scope}
\begin{scope}[shift={(11,-6)}]
\node at (-2.5,2.25) {\Dp} ; \node at (-2.5,1.75) {$\longrightarrow$} ;
\node (v) [b,label=above:${}_{}v$] at (2,4) {};
\node (a2) [b,label=left:$a_2$\,] at (0,2) {};
\node (a1) [b,label=below left:$a_1$] at (0,0) {} ;
\node (a4) [b,label=below right:$a_4$] at (2,0)  {};
\node (a3) [b,label=right:\,$a_3$] at (2,2)  {};
\node (u) [b,label=above:${}_{}u$] at (0,4) {};
\draw (a2) -- (v) -- (a3) (a1) -- (a3) (v) -- (u) -- (a2) (a1)--(a4) (a2) -- (a4) (u)--(a3);
\draw [dotted] (a2) -- (a3) (a2) -- (a1)(a1) -- (a3) (a4) -- (a3);
\draw [dotted] (u)edge[bend right=60](a1) (a4)edge[bend right=60](v);
\end{scope}
\end{tikzpicture}
\end{center}

\item[(VI)]
Finally, suppose that the switch has Type~(c), so exactly one diagonal $a_ja_{5-j}$
is present, and that $A_i = A_{5-i}$, and that $G[A_i]$ has no edges,
where $\{i,j\} = \{ 1,2\}$,
Again, by relabelling if necessary, we assume that diagonal $a_1a_4$ is present,
and write $A=A_2 = A_3$.
Note that $|A|\geq 2$ as $d_D(a_2) = d_D(a_3)=1$.
Let $(u,v)$ be the lexicographically-least pair of distinct
elements in $A$, and let $w$ be the smallest-labelled vertex in $\Nb(v) \setminus D$,
which exists as the minimum degree is at least~3.
Then $w\not\in A = A_3$  as $v\in A$ and $vw\in E$.
We can construct the following simulation path of length~4.
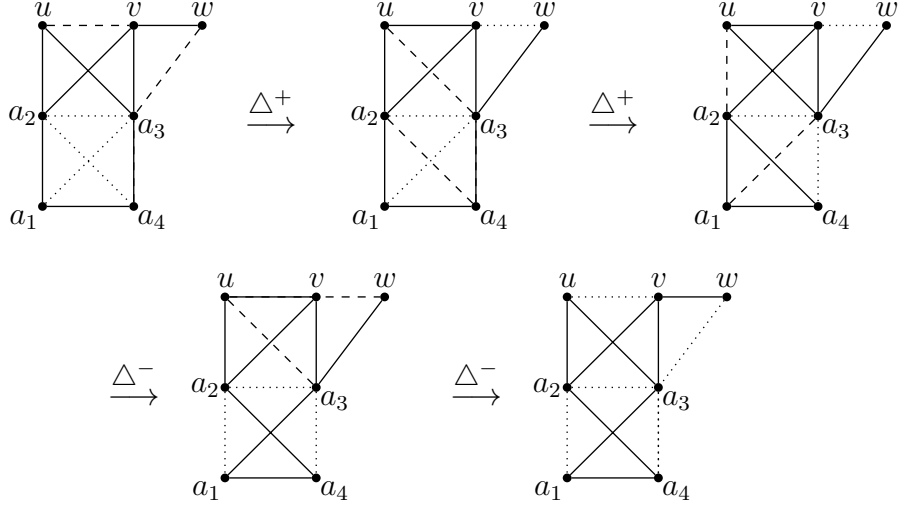
\begin{center}
\begin{tikzpicture}[scale=0.60,inner sep=0pt]
\begin{scope}[shift={(0,0)}]
\node (w) [b,label=above:${}_{}w$] at (3.5,4) {};
\node (v) [b,label=above:${}_{}v$] at (2.0,4) {};
\node (a2) [b,label=left:$a_2$] at (0,2) {};
\node (a1) [b,label=below left:$a_1$] at (0,0) {} ;
\node (a4) [b,label=below right:$a_4$] at (2,0)  {};
\node (a3) [b,label=below right:$a_3$] at (2,2)  {};
\node (u) [b,label=above:${}_{}u$] at (0.0,4) {};
\draw  (u) -- (a2) -- (a1) (a1)--(a4) (a2) -- (v) -- (a3) (a3)-- (a4)  (w) -- (v)  (u)--(a3);
\draw [dashed]  (a4) -- (a3) (u) -- (v) (w) -- (a3);
\draw[dotted] (a1)--(a3)--(a2)--(a4);
\node at (5,2.25) {\small \Dp} ; \node at (5,1.75) {$\longrightarrow$} ;
\end{scope}
\begin{scope}[shift={(7.5,0)}]
\node (w) [b,label=above:${}_{}w$] at (3.5,4) {};
\node (v) [b,label=above:${}_{}v$] at (2.0,4) {};
\node (a2) [b,label=left:$a_2$] at (0,2) {};
\node (a1) [b,label=below left:$a_1$] at (0,0) {} ;
\node (a4) [b,label=below right:$a_4$] at (2,0)  {};
\node (a3) [b,label=below right:$a_3$] at (2,2)  {};
\node (u) [b,label=above:${}_{}u$] at (0.0,4) {};
\draw  (a2) -- (a1) (a1)--(a4)(a2) -- (v) --(a3)-- (a4)   (u) -- (a2) (u) -- (v) (w) -- (a3);
\draw [dashed] (a2) -- (a4) (a4) -- (a3) (u) -- (a3);
\draw[dotted] (v) -- (w) (a2) -- (a3) (a1) -- (a3) ;
\node at (5,2.25) {\small \Dp} ; \node at (5,1.75) {$\longrightarrow$} ;
\end{scope}
\begin{scope}[shift={(15,0)}]
\node (w) [b,label=above:${}_{}w$] at (3.5,4) {};
\node (v) [b,label=above:${}_{}v$] at (2.0,4) {};
\node (a2) [b,label=left:$a_2$] at (0,2) {};
\node (a1) [b,label=below left:$a_1$] at (0,0) {} ;
\node (a4) [b,label=below right:$a_4$] at (2,0)  {};
\node (a3) [b,label=below right:$a_3$] at (2,2)  {};
\node (u) [b,label=above:${}_{}u$] at (0,4) {};
\draw  (a2) -- (a1) (v) -- (a2) (w) -- (a3) (a1)--(a4) (a2) -- (a4)(v) -- (a3) (v) -- (u)  (u)--(a3);
\draw [dotted] (a3)-- (a4) (a2) -- (a3) (v) -- (w);
\draw[dashed] (a2) -- (u) (a1) -- (a3) ;
\end{scope}
\begin{scope}[shift={(4,-6)}]
\node at (-2,2.25) {\small \Dm} ; \node at (-2,1.75) {$\longrightarrow$} ;
\node (w) [b,label=above:${}_{}w$] at (3.5,4) {};
\node (v) [b,label=above:${}_{}v$] at (2.0,4) {};
\node (a2) [b,label=left:$a_2$] at (0,2) {};
\node (a1) [b,label=below left:$a_1$] at (0,0) {} ;
\node (a4) [b,label=below right:$a_4$] at (2,0)  {};
\node (a3) [b,label=below right:$a_3$] at (2,2)  {};
\node (u) [b,label=above:${}_{}u$] at (0.0,4) {};
\draw  (u) -- (a2) (w) -- (a3) (a1)--(a4) (a2) -- (a4)(a2) -- (v) (a1) -- (a3) (v) -- (u)  (v)--(a3);
\draw [dashed]  (u) -- (w) (u) -- (a3) ;
\draw[dotted] (a2) -- (a1) (a2) -- (a3) (a3)-- (a4) ;
\end{scope}
\begin{scope}[shift={(11.5,-6)}]
\node at (-2,2.25) {\small \Dm} ; \node at (-2,1.75) {$\longrightarrow$} ;
\node (w) [b,label=above:${}_{}w$] at (3.5,4) {};
\node (v) [b,label=above:${}_{}v$] at (2,4) {};
\node (a2) [b,label=left:$a_2$] at (0,2) {};
\node (a1) [b,label=below left:$a_1$] at (0,0) {} ;
\node (a4) [b,label=below right:$a_4$] at (2,0)  {};
\node (a3) [b,label=below right:$a_3$] at (2,2)  {};
\node (u) [b,label=above:${}_{}u$] at (0,4) {};
\draw  (u) -- (a2) (a1)--(a4) (a2) -- (a4)(a2) -- (v) (a1) -- (a3) (v) -- (w) (v) -- (a3)  (u)--(a3);
\draw [dotted] (a4) -- (a3) (a2) -- (a3) (a3)-- (a4) (a2) -- (a1) (u) -- (v) (w) -- (a3);
\end{scope}
\end{tikzpicture}
\end{center}
\end{itemize}
To complete the proof, observe that (I), (III) and (IV) together cover all Type~(b) switches,
while Type~(c) switches are covered by combining (I), (IV), (V) and (VI).
	This proves the lemma under assumption (i).
The proof for assumption (ii) is covered by (I) and (II) together,
while (IV) completes the proof for assumption (iii).
\end{proof}

For the remainder of the proof we assume that the assumptions of
Lemma~1 do not hold.  Hence the switch has Type~(a),
that is, no diagonals are present. Furthermore, no two vertices in $D$
have a common neighbour outside $D$, and there is no triangle of the
form $a_j u w$ with 
$\{u,w\}\cap D = \es$, where $j\in [4]$.

In this situation it is possible that the switch edges $a_1a_2$, $a_3a_4$
lie in different components of $G$, so that $H$ has one fewer components
that $G$. Only a \dsm\ can reduce the number of components, and so our
strategy is to ``plant'' a triangle of the form $a_j,u, w$ with
$u,w\not\in D$ distinct.  By assumption, $ua_{5-j}$ is not an edge as
$A_j\cap A_{5-j}=\es$.  Hence, after planting the
triangle, we can simulate the switch using the simulation path
described in part (IV) of Lemma~\ref{lem:pivot}, then ``unplant'' the
planted triangle.  There is one exceptional situation where we will not
plant a triangle, but instead will simulate the switch $(G,H)$ with
a simulation path of length three.  (This is case (IXa) of Lemma~\ref{lem:no-pivot} below.)

We will make use of the following lemma to plant a triangle on some
element of $D$. This lemma will be applied to graphs in $\Gnd$, or
to graphs in $\Gnd$ with one edge removed.

\bigskip

\begin{lem}\label{lem:plant}
Let $G$ be a graph on the vertex set $[n]$.
Suppose $v\in [n]$ has a set of three distinct neighbours $R=\{r_1,r_2,r_3\}$, with $\deg(r_i)\geq 2$ for $i=1,2,3$.
Assume that there is no triangle in $G$ which contains $v$ and an element of $R$.
\begin{enumerate}
\item[\emph{(i)}]
Then there is a \dsp\ $(G,\widetilde{G})$ such that $\widetilde{G}$ contains a triangle $T$ with $v\in T$ and $R\cap T\neq\es$.
\item[\emph{(ii)}]
Furthermore, we can insist that $r_1\in T$ unless
for $j=2,3$ there is a 5-cycle
which contains the path $r_1 v r_j$, but no 4-cycle in $G$
contains this path.
\end{enumerate}
\end{lem}

\begin{proof}
Let $\ell$ be the length of the shortest cycle in $G$ which contains a path
of the form $r_i v r_j$, where $r_i, r_j$ are distinct elements of $R$.
If no such cycle exists, set $\ell=\infty$. By assumption, $\ell\geq 4$. First suppose that $\ell=4$. After relabelling the vertices in $R$ if necessary, assume that the cycle contains the path $r_1vr_2$ together with some vertex $w$. (For definiteness, let $w$ be the least-labelled common neighbour of $r_1$ and $r_2$.) Note that $vw$ and $r_2r_3$ are not edges of $G$, by assumption.
Then we can plant a triangle on $v, r_1, w$ as follows:
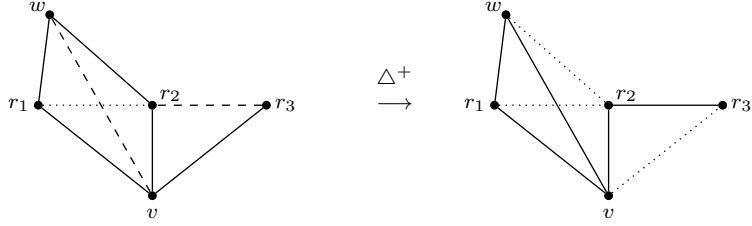
\begin{center}
\begin{tikzpicture}[xscale=0.75,yscale=0.75,inner sep=0pt,font=\scriptsize]
\begin{scope}[shift={(0,0)}]
\node (v) [b,label=below:$\strut v$] at (0,0.4) {};
\node (u1) [b,label=left:$r_1\,$] at (-2,2) {};
\node (u2) [b,label=above right:$\,r_2$] at (0,2) {} ;
\node (u3) [b,label=right:$\,r_3$] at (2,2)  {};
\node (w) [b,label=above left:$w$] at (-1.8,3.6) {};
\draw (v)--(u1) (v)--(u2) (v)--(u3) (u1)--(w) (u2)--(w) ;
\draw[dashed] (v)--(w) (u3)--(u2) ;
\draw[dotted] (u1)--(u2);
\node at (4.25,2.5) {\Dp} ; \node at (4.25,2) {$\longrightarrow$} ;
\end{scope}
\begin{scope}[shift={(8,0)}]
\node (v) [b,label=below:$\strut v$] at (0,0.4) {};
\node (u1) [b,label=left:$r_1\,$] at (-2,2) {};
\node (u2) [b,label=above right:$\,r_2$] at (0,2) {};
\node (u3) [b,label=right:$\,r_3$] at (2,2)  {};
\node (w) [b,label=above left:$w$] at (-1.8,3.6) {};
\draw (v)--(u1) (v)--(u2) (u1)--(w) (v)--(w) (u3)--(u2) ;
\draw[dotted] (u1)--(u2) (w)--(u2) (v)--(u3) ;
\end{scope}
\end{tikzpicture}
\end{center}
Next suppose that $\ell=5$. After relabelling elements of $R$ if necessary,
we can assume that the 5-cycle is $v r_1 w_1 w_2 r_2 v$, where $w_1,w_2\not\in \{v\}\cup R$ are distinct. For definiteness,
choose $(w_1,w_2)$ to be lexicographically least among all possibilities.
Note that $vw_1$ is a non-edge in $G$ by assumption, and $w_2r_3$ is a non-edge
in $G$ as $\ell\neq 4$. We can ``plant'' a triangle as follows:
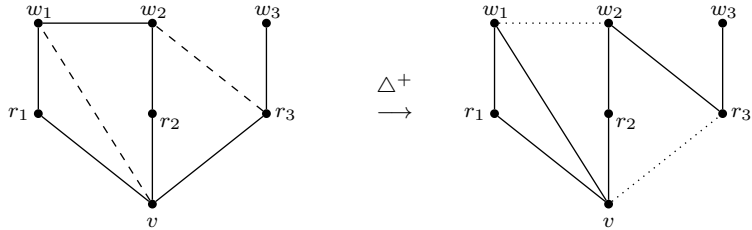
\begin{center}
\begin{tikzpicture}[xscale=0.75,yscale=0.75,inner sep=0pt,font=\scriptsize]
\node (v) [b,label=below:$\strut v$] at (0,0.4) {};
\node (u1) [b,label=left:$r_1\,$] at (-2,2) {};
\node (u2) [b,label=below right:$\,r_2$] at (0,2) {} ;
\node (u3) [b,label=right:$\,r_3$] at (2,2)  {};
\node (w1) [b,label=above:$\,w_1$] at (-2,3.6) {};
\node (w2) [b,label=above:$\,w_2$] at (0,3.6) {};
\node (w3) [b,label=above:$\,w_3$] at (2,3.6) {};
\draw (v)--(u1) (v)--(u2) (v)--(u3) (u1)--(w1) (u2)--(w2) (u3)--(w3) (w1)--(w2);
\draw[dashed] (v)--(w1) (u3)--(w2);
\node at (4.25,2.5) {\Dp} ; \node at (4.25,2) {$\longrightarrow$} ;
\begin{scope}[shift={(8,0)}]
\node (v) [b,label=below:$\strut v$] at (0,0.4) {};
\node (u1) [b,label=left:$r_1\,$] at (-2,2) {};
\node (u2) [b,label=below right:$\,r_2$] at (0,2) {} ;
\node (u3) [b,label=right:$\,r_3$] at (2,2)  {};
\node (w1) [b,label=above:$\,w_1$] at (-2,3.6) {};
\node (w2) [b,label=above:$\,w_2$] at (0,3.6) {};
\node (w3) [b,label=above:$\,w_3$] at (2,3.6) {};
\draw (v)--(u1) (v)--(u2) (u3)--(w2)--(u2) (u3)--(w3) (u1)--(w1)--(v) ;
\draw[dotted] (v)--(u3) (w1)--(w2);
\end{scope}
\end{tikzpicture}
\end{center}
Now suppose that $\ell\geq 6$. After relabelling vertices in $R$
if necessary, suppose that we have a path $w_1r_1vr_2w_2$ in $G$,
	where $w_1,w_2\not\in \{v\}\cup R$ are distinct.
	(For definiteness, choose $(w_1,w_2)$ to be lexicographically
	least among all options.)
Note that $r_1r_2$ is a non-edge in $G$ by assumption, and $w_1w_2$
is a non-edge in $G$ as $\ell\neq 5$.  Then we can plant a triangle
as follows:
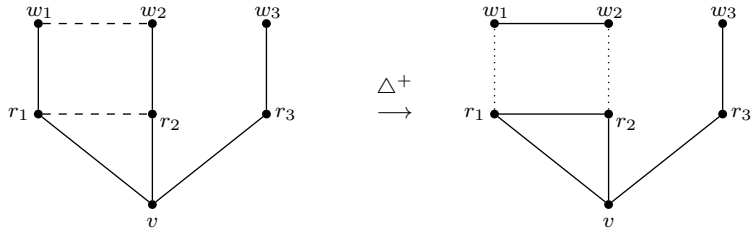
\begin{center}
\begin{tikzpicture}[xscale=0.75,yscale=0.75,inner sep=0pt,font=\scriptsize]
\begin{scope}[shift={(0,0)}]
\node (v) [b,label=below:$\strut v$] at (0,0.4) {};
\node (u1) [b,label=left:$r_1\,$] at (-2,2) {};
\node (u2) [b,label=below right:$\,r_2$] at (0,2) {} ;
\node (u3) [b,label=right:$\,r_3$] at (2,2)  {};
\node (w1) [b,label=above:$\,w_1$] at (-2,3.6) {};
\node (w2) [b,label=above:$\,w_2$] at (0,3.6) {};
\node (w3) [b,label=above:$\,w_3$] at (2,3.6) {};
\draw (v)--(u1) (v)--(u2) (v)--(u3) (u1)--(w1) (u2)--(w2) (u3)--(w3) ;
\draw[dashed] (u1)--(u2) (w1)--(w2);
\node at (4.25,2.5) {\Dp} ; \node at (4.25,2) {$\longrightarrow$} ;
\end{scope}
\begin{scope}[shift={(8,0)}]
\node (v) [b,label=below:$\strut v$] at (0,0.4) {};
\node (u1) [b,label=left:$r_1\,$] at (-2,2) {};
\node (u2) [b,label=below right:$\,r_2$] at (0,2) {} ;
\node (u3) [b,label=right:$\,r_3$] at (2,2)  {};
\node (w1) [b,label=above:$\,w_1$] at (-2,3.6) {};
\node (w2) [b,label=above:$\,w_2$] at (0,3.6) {};
\node (w3) [b,label=above:$\,w_3$] at (2,3.6) {};
\draw (v)--(u1) (v)--(u2) (v)--(u3)  (u1)--(u2) (u3)--(w3) (w1)--(w2) ;
\draw[dotted] (u1)--(w1) (u2)--(w2);
\end{scope}
\end{tikzpicture}
\end{center}
In all cases, we have (uniquely) identified a \dsp\ which produces
a graph $\widetilde{G}$ with the same degree sequence as $G$
which contains a triangle satisfying
the stated conditions. This proves~(i).

To prove the second statement, adapt the above proof by letting $\ell$ be the
length of the shortest
cycle which contains the path $r_1 v r_j$ for some $j\in \{2,3\}$.
If $\ell\neq 5$ then the \ds\ described above creates a triangle $T$ with
$r_1\in T$, as required. If $\ell=5$ then the \ds\ described above works if
and only if $w_2r_3$ is not an edge, assuming (after swapping the labels of $r_2$ and $r_3$ if necessary) that the 5-cycle $v r_1 w_1 w_2 r_2 v$ is present.
This completes the proof of~(ii).
\end{proof}

To describe how to remove a planted triangle, we need some terminology.
Let $(G,\widetilde{G})$ be a \ds, where $\widetilde{G}$ is obtained from $G$ by deleting the edges
$e_1, e_2$ and inserting the edges $e_3,e_4$.  If $\widetilde{H}$ is any graph which
contains the edges $e_3,e_4$ and does not contain $e_1,e_2$, then the
\emph{inverse} of the \ds\ $(G,\widetilde{G})$, applied to $\widetilde{H}$, is the \ds\ $(\widetilde{H},H)$
where $H$ is
obtained from $\widetilde{H}$ by deleting the edges $e_3,e_4$ and inserting the edges
$e_1,e_2$.  This matches the usual notion of an inverse operation when
$\widetilde{H}=\widetilde{G}$, since then $H=G$, but here we need to apply inverse \ds es when
$\widetilde{H}\neq \widetilde{G}$.

For us, $(G,\widetilde{G})$ will be a \dsp\ which plants a triangle.
Then we will perform two or three \ds es from $\widetilde{G}$ to reach a graph $\widetilde{H}$ which
contains $e_3,e_4$ but not $e_1,e_2$. Finally we will remove the planted
triangle by applying the inverse of $(G,\widetilde{G})$ to $\widetilde{H}$.
There is one exception, which is case~(IXa) below, where we do not
plant a triangle but instead simulate the required switch with a simulation path of length three.

\begin{lem}
\label{lem:no-pivot}
Suppose that $(G,H)$ is a switch which removes the edges $a_1a_2$, $a_3a_4$
and replaces them with the edges $a_1a_3$, $a_2a_4$. Further suppose
that the following conditions all hold:
\begin{itemize}[topsep=0pt,itemsep=0pt]
\item The switch has Type~\emph{(a)} (no diagonals present),
\item no pair of vertices
in $D=\{a_1,a_2,a_3,a_4\}$ have a common neighbour, and
\item there is no
triangle on vertices $a_j,u,w$ with 
$\{u,w\}\cap D = \es$, where $j\in [4]$.
\end{itemize}
Then there is a \ds\ simulation path of length
at most five which simulates the switch $(G,H)$.
\end{lem}

\begin{proof}
We will number the cases which arise when defining the simulation paths.
This will be needed later in the proof of rapid mixing.
Since Lemma~\ref{lem:pivot} has cases (I)--(VI), here our
numbering starts from (VII).
We treat these cases in order, so we only move to~(VIII) if~(VII) does not
hold, and so on.
\begin{itemize}
\item[(VII)]
First suppose that after at least one of $a_1,a_2,a_3,a_4$ has degree at least~4.
Choose the least-labelled element of $D$ which has degree at least~4: after
relabelling the elements of $D$ if necessary, we can assume that $a_1$ has degree
at least~4.
We apply Lemma~\ref{lem:plant}(i) with $v=a_1$, and $R$ a set of three neighbours of $a_1$ other than $a_2$. This identifies a \dsp\ $(G,\widetilde{G})$ where $\widetilde{G}$ contains a triangle of the form $a_1,u,w$.  From $\widetilde{G}$ we can perform the two \ds es described in Lemma~\ref{lem:pivot}, part (IV), to delete the edges $a_1a_2, a_3a_4$ and replace them with $a_1a_3$, $a_1a_4$ using two \ds es.  From the resulting graph $\widetilde{H}$ we apply the inverse of the \dsp\ $(G,\widetilde{G})$.  This ``unplants'' the planted triangle and takes us to $H$.
This gives a simulation path of length~4.
\item[(VIII)]
Next, assume that $\deg(a_i)=3$ for all $i\in [4]$ and that there is a path
of length at most~2 in $G\setminus D$ between some element of $A_j$ and some
element of $A_{2+j}$, for some $j\in \{1,2\}$. Take a shortest
such path, and by relabelling if necessary assume that $j=1$.
Let $A_1 = \{u_1,u_2\}$ and $A_3 = \{v_1,v_2\}$ and by relabelling
these vertices if necessary, assume that the path has endvertices $u_1$, $v_1$.
Note that $u_2\not\in A_4$ and $v_1\not\in A_2$, since no two elements of
$D$ have a common neighbour (by assumption).

(VIIIa):\ If the path has length~1 then the edge $u_1v_1$ is present and we can plant a triangle as follows.
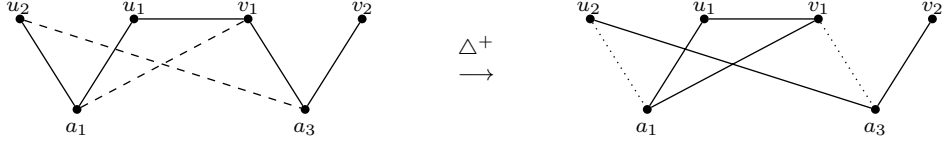
\begin{center}
\begin{tikzpicture}[xscale=0.75,yscale=0.75,inner sep=0pt,font=\scriptsize]
\begin{scope}[shift={(0,0)}]
\node (a1) [b,label=below:$\strut a_1$] at (0,0.4) {};
\node (u2) [b,label=above:${}_{}u_2\,$] at (-1,2) {};
\node (u1) [b,label=above:${}_{}u_1$] at (1,2) {} ;
\node (a3) [b,label=below:\strut $a_3$] at (4,0.4)  {};
\node (v2) [b,label=above:${}_{}v_2$] at (5,2) {};
\node (v1) [b,label=above:${}_{}v_1$] at (3,2) {};
\draw (u1)--(a1)--(u2) (v1)--(a3)--(v2) (u1)--(v1) ;
\draw[dashed] (u2)--(a3) (a1)--(v1) ;
\node at (7,1.5) {\Dp} ; \node at (7,1) {$\longrightarrow$} ;
\end{scope}
\begin{scope}[shift={(10,0)}]
\node (a1) [b,label=below:$\strut a_1$] at (0,0.4) {};
\node (u2) [b,label=above:${}_{}u_2\,$] at (-1,2) {};
\node (u1) [b,label=above:${}_{}u_1$] at (1,2) {} ;
\node (a3) [b,label=below:\strut $a_3$] at (4,0.4)  {};
\node (v2) [b,label=above:${}_{}v_2$] at (5,2) {};
\node (v1) [b,label=above:${}_{}v_1$] at (3,2) {};
\draw (u1)--(a1) (a3)--(v2) (u1)--(v1) (u2)--(a3) (a1)--(v1);
\draw[dotted] (u2)--(a1) (v1)--(a3);
\end{scope}
\end{tikzpicture}
\end{center}
We have a triangle $a_1,u_1,v_2$ which allows us to perform the two \ds es as
described in Lemma~\ref{lem:pivot}, part (IV), to simulate the required switch.
From the resulting graph, we apply the inverse of the triangle-planting \ds, to
``unplant'' the triangle and produce $H$.
This gives a simulation path of length~4.

(VIIIb): If the path has length~2 then (after possibly relabelling as described
above), $u_1$ and $v_1$ have a common neighbour $w$, where we choose $w$ to
be the least-labelled common neighbour of $u_1$, $v_1$.
Note that $a_3w \notin E$ because $\deg(a_3)=3$, and
$u_2v_1\notin E$ since otherwise it would be an edge between the neighbourhoods of $a_1,a_3$. Then we can plant a triangle as follows:
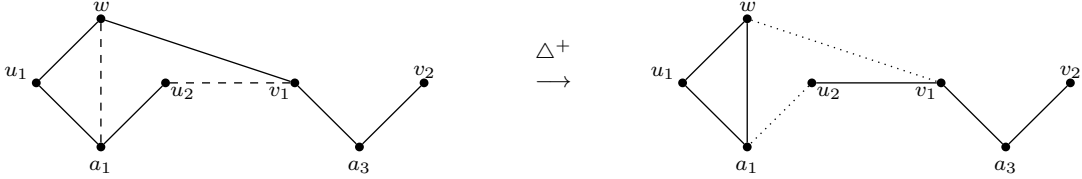
\begin{center}
\begin{tikzpicture}[xscale=0.85,yscale=0.85,inner sep=0pt,font=\scriptsize]
\begin{scope}[shift={(0,0)}]
\node (a1) [b,label=below:$\strut a_1$] at (0,1.0) {};
\node (u1) [b,label=above left:${}_{}u_1\,$] at (-1,2) {};
\node (u2) [b,label=below right:${}_{}u_2$] at (1,2) {} ;
\node (a3) [b,label=below:\strut $a_3$] at (4,1.0)  {};
\node (v2) [b,label=above:${}_{}v_2$] at (5,2) {};
\node (v1) [b,label=below left:${}_{}v_1$] at (3,2) {};
\node (w) [b,label=above:${}_{}w$] at (0,3) {} ;
\draw (u1)--(a1)--(u2) (v1)--(a3)--(v2) (u1)--(w)--(v1) ;
\draw[dashed] (w)--(a1) (u2)--(v1) ;
\node at (7,2.5) {\Dp} ; \node at (7,2) {$\longrightarrow$} ;
\end{scope}
\begin{scope}[shift={(10,0)}]
\node (a1) [b,label=below:$\strut a_1$] at (0,1.0) {};
\node (u1) [b,label=above left:${}_{}u_1\,$] at (-1,2) {};
\node (u2) [b,label=below right:${}_{}u_2$] at (1,2) {} ;
\node (a3) [b,label=below:\strut $a_3$] at (4,1.0)  {};
\node (v2) [b,label=above:${}_{}v_2$] at (5,2) {};
\node (v1) [b,label=below left:${}_{}v_1$] at (3,2) {};
\node (w) [b,label=above:${}_{}w$] at (0,3) {} ;
\draw (u1)--(a1) (v1)--(a3)--(v2) (u1)--(w) (w)--(a1) (u2)--(v1);
\draw[dotted] (w)--(v1) (u2)--(a1);
\end{scope}
\end{tikzpicture}
\end{center}
Now we have a triangle $a_1,u_1,w$ which we can use to perform the two \dsp es
as specified in Lemma~\ref{lem:pivot}, case~(IV).  From the resulting graph, we
apply the inverse of the triangle-planting \ds, to ``unplant''  this triangle
and produce $H$. This gives a simulation path of length~4.
\item[(IX)]
Here we assume that $\deg(a_i)=3$ for all $i\in [4]$ and that there is
no path of length at most 2 between an element of $A_j$ and an element of
$A_{2+j}$, for $j=1,2$.  Let $A_1 = \{u_1,u_2\}$.

(IXa)
First suppose that there is an edge between two neighbours of $u_j$,
and by relabelling if necessary we can assume that $j=1$.  Then
$w_1w_2\in E$ for some $w_1, w_2\in \Nb(u_1)$ which are distinct from $a_1$.
(For definiteness, let $(w_1,w_2)$ be the lexicographically-least such pair.)
In this subcase,
we do not plant a triangle: instead we use the following simulation path of length~3, where $v$ is the least-labelled element of $A_3$.
Note that $u_1a_3\not\in E$ since $A_1\cap A_3 =\es$,
and $vw_2\not\in E$ since $u_1$ and $v$ are at distance at least~3 in $G$.
\begin{figure}
\begin{center}
\begin{tikzpicture}[xscale=0.75,yscale=0.75,inner sep=0pt,font=\scriptsize]
\begin{scope}[shift={(0,0)}]
\node (u1) [b,label=left:$u_1\,$] at (0,2) {};
\node (w1) [b,label=left:$w_1\,$] at (-0.8,3.5)  {};
\node (w2) [b,label=right:$\,w_2$] at (0.8,3.5) {};
\node (a2) [b,label=left:$a_2\,$] at (0,-1.2) {} ;
\node (a1) [b,label=left:$a_1\,$] at (0,0.4) {};
\node (a4) [b,label=right:$\,a_4$] at (2,-1.2) {} ;
\node (a3) [b,label=right:$\,a_3$] at (2,0.4) {};
\node (v) [b,label=right:$\,v$] at (2,2) {};
\draw (a2)--(a1)--(u1)--(w1)--(w2)--(u1) (a4)--(a3)--(v);
\draw [dashed] (a3)--(u1) (v)--(w2) ;
\draw[dotted] (a2)--(a3)--(a1)--(a4)--(a2);
\node at (4,1.5) {\Dm}; \node at (4,1.0) {$\longrightarrow$};
\end{scope}
\begin{scope}[shift={(6,0)}]
\node (u1) [b,label=left:$u_1\,$] at (0,2) {};
\node (w1) [b,label=left:$w_1\,$] at (-0.8,3.5)  {};
\node (w2) [b,label=right:$\,w_2$] at (0.8,3.5) {};
\node (a2) [b,label=left:$a_2\,$] at (0,-1.2) {} ;
\node (a1) [b,label=left:$a_1\,$] at (0,0.4) {};
\node (a4) [b,label=right:$\,a_4$] at (2,-1.2) {} ;
\node (a3) [b,label=right:$\,a_3$] at (2,0.4) {};
\node (v) [b,label=right:$\,v$] at (2,2) {};
\draw (a2)--(a1)--(u1)--(w1)--(w2) (a4)--(a3) (a3)--(u1) (v)--(w2);
\draw [dashed] (a1)--(a3) (a2)--(a4);
\draw[dotted] (a2)--(a3) (a1)--(a4) (u1)--(w2) (a3)--(v);
\end{scope}
\begin{scope}[shift={(5,-6)}]
\node at (-2,1.5) {\Dp}; \node at (-2,1.0) {$\longrightarrow$};
\node (u1) [b,label=left:$u_1\,$] at (0,2) {};
\node (w1) [b,label=left:$w_1\,$] at (-0.8,3.5)  {};
\node (w2) [b,label=right:$\,w_2$] at (0.8,3.5) {};
\node (a2) [b,label=left:$a_2\,$] at (0,-1.2) {} ;
\node (a1) [b,label=left:$a_1\,$] at (0,0.4) {};
\node (a4) [b,label=right:$\,a_4$] at (2,-1.2) {} ;
\node (a3) [b,label=right:$\,a_3$] at (2,0.4) {};
\node (v) [b,label=right:$\,v$] at (2,2) {};
\draw (a3)--(a1)--(u1)--(w1)--(w2) (u1)--(a3) (v)--(w2) (a2)--(a4);
\draw [dashed]  (a3)--(v) (u1)--(w2) ;
\draw[dotted] (a2)--(a3)--(a4)--(a1)--(a2);
\begin{scope}[shift={(6,0)}]
\node at (-2,1.5) {\Dp}; \node at (-2,1.0) {$\longrightarrow$};
\node (u1) [b,label=left:$u_1\,$] at (0,2) {};
\node (w1) [b,label=left:$w_1\,$] at (-0.8,3.5)  {};
\node (w2) [b,label=right:$\,w_2$] at (0.8,3.5) {};
\node (a2) [b,label=left:$a_2\,$] at (0,-1.2) {} ;
\node (a1) [b,label=left:$a_1\,$] at (0,0.4) {};
\node (a4) [b,label=right:$\,a_4$] at (2,-1.2) {} ;
\node (a3) [b,label=right:$\,a_3$] at (2,0.4) {};
\node (v) [b,label=right:$\,v$] at (2,2) {};
\draw  (v)--(a3)--(a1)--(u1)--(w1)--(w2)--(u1) (a2)--(a4);
\draw[dotted] (a2)--(a3)--(a4)--(a1)--(a2) (u1)--(a3) (v)--(w2);
\end{scope}
\end{scope}
\end{tikzpicture}
\end{center}
\end{figure}

(IXb) Now suppose that there is no edge between any two neighbours of $u_j$
for $j\in \{1,2\}$.
Let $w_1,w_2$ be the two least-labelled neighbours of $u_1$ distinct from $a_1$.
	We want to apply Lemma~\ref{lem:plant}(ii) to the graph
	$G^{(-)}:=G - \{ a_1a_2\}$ with the edge $a_1a_2$ deleted, setting $v=u_1$ and
	$R = \{ a_1,w_1, w_2\}$ with $r_1=a_1$.
	Suppose that
Lemma~\ref{lem:plant}(ii) succeeds in identifying a \dsp\ $(G^{(-)},\widetilde{G}^{(-)})$ such that $\widetilde{G}^{(-)}$
contains a triangle on the vertices $a_1, u_1, z$ for some vertex $z\not\in D$.
Let $\widetilde{G}$ denote the graph on $\Gnd$ obtained by
inserting the edge $a_1a_2$ in $\widetilde{G}^{(-)}$.
Then $(G,\widetilde{G})$ is a \ds\ which can be applied to $G$.
From $\widetilde{G}$, we can perform the two \ds es
described in Lemma~\ref{lem:pivot} case (IV), to exchange the switch edges.
Then from the resulting graph, apply the inverse of the \ds\
$(G,\widetilde{G})$
to ``unplant'' the triangle and produce the graph $H$.  This gives a
simulation path of length~4.

(IXc)  Finally we assume that there is no edge between any two neighbours of
$u_j$, for $j\in \{1,2\}$, but that Lemma~\ref{lem:plant}(ii) did not
succeed when applied to the graph $G^{(-)}:= G - \{a_1a_2\}$ with
$v=u_1$, $R= \{a_1,w_1,w_2\}$ and $r_1=a_1$.
(As above, we assume that $w_1,w_2$ are the two least-labelled
neighbours of $u_1$ distinct from $a_1$.)
In particular, this implies that $G^{(-)}$ has a 5-cycle $C$ which
contains the
path $a_1u_1w_1$, but $G^{(-)}$ does not have any 4-cycle which
contains the path $a_1u_1w_j$ for $j\in \{1,2\}$.
The same properties are true for $G$.
Since $a_1$ has degree~2 in $G^{(-)}$,
the 5-cycle $C$ must be $a_1u_1w_1zu_2$ for some vertex $z$.
Note that $u_1z$ cannot be an edge of $G$, or we are in case (IXa),
and $w_2u_2$ cannot be and edge of $G$ as this would give
rise to a 4-cycle $a_1u_1w_2u_2$ in $G^{(-)}$ containing the path $a_1u_1w_2$.
Hence we can perform the following \ds\ in $G$ to plant a triangle
on $u_1, w_1, z$.

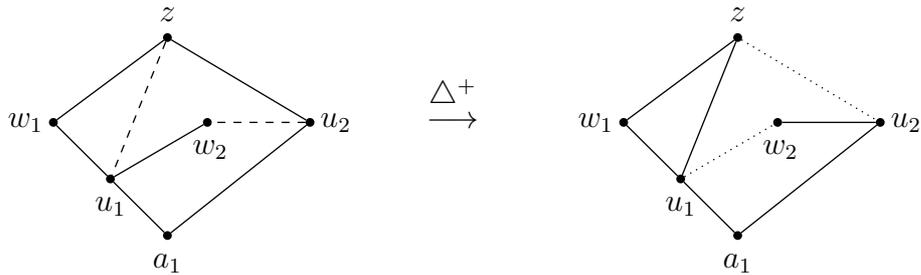
\begin{center}
\begin{tikzpicture}[xscale=0.75,yscale=0.75,inner sep=0pt] 
\begin{scope}
\node (u1) [b,label=below:$\strut u_1$] at (-1,1) {};
\node (a1) [b,label=below:$\strut a_1$] at (0,0) {};
\node (w1) [b,label=below:$\strut \, w_2$] at (0.7,2) {} ;
\node (w2) [b,label=left:$w_1\,$] at (-2,2)  {};
\node (u2) [b,label=right:$\,u_2$] at (2.5,2) {};
\node (z) [b,label=above:$z\strut$] at (0,3.5) {};
\draw (u1)--(a1) (u1)--(w1) (u1)--(w2) (a1)--(u2) (z)--(w2)  (u2)--(z);
\draw[dashed] (u2)--(w1) (z) -- (u1);
\node at (5,2.5) {\Dp} ; \node at (5,2) {$\longrightarrow$} ;
\end{scope}
\begin{scope}[shift={(10,0)}]
\node (u1) [b,label=below:$\strut u_1$] at (-1,1) {};
\node (a1) [b,label=below:$\strut a_1$] at (0,0) {};
\node (w1) [b,label=below:$\strut \, w_2$] at (0.7,2) {} ;
\node (w2) [b,label=left:$w_1\,$] at (-2,2)  {};
\node (u2) [b,label=right:$\,u_2$] at (2.5,2) {};
\node (z) [b,label=above:$z\strut$] at (0,3.5) {};
\draw (u1)--(a1) (z) -- (u1) (w1) -- (u2) (u1)--(w2)--(z) (a1)--(u2);
\draw[dotted] (u1)--(w1) (z) -- (u2);
\end{scope}
\end{tikzpicture}
\end{center}

Then we have the situation in (IXa), where $u_2$ and its neighbours are temporarily unimportant. The three \ds es given in case (IXa) perform the required switch, and from the resulting graph we apply the inverse of the \ds\ which planted $u_1w_1w_2$, reaching the graph $H$.  This gives a simulation path of length five.\qedhere
\end{itemize}
\end{proof}

Now Theorem~\ref{thm:five} follows immediately by 
combining Lemma~\ref{lem:pivot} and Lemma~\ref{lem:no-pivot}.
Finally, Theorem~\ref{thm:irreducible} is obtained from Theorem~\ref{thm:five}
using
the fact that the switch chain is irreducible on $\Gnd$ for all graphical degree 
sequences $\dsq$ (see~\cite{fulkerson}).


\section{Rapid mixing of the modified Metropolis \ds\ chain}\label{sec:rapid}

In this section we prove Theorem~\ref{thm:rapid}
using the two-stage direct canonical path construction method,
restated above as Theorem~\ref{thm:CDGH}.
Recall the definitions from Section~\ref{sec:background-paths}.
Let $\cM$ denote the modified Metropolis \ds\ chain (see Figure~\ref{fig:modified-triangle-switch-code}) and let $\cM'$ denote the modified Metropolis switch chain 
(see Figure~\ref{fig:switch-metropolis}).  

In Section~\ref{sec:irreducible} we constructed a set
\[ \Sigma=\{\sigma_{GH} \mid (G,H)\in E(\cM')\}\]
of $(\cM,\cM')$-simulation paths.
Recall the parameters defined in (\ref{eq:parameters}).
First we prove an upper bound on the parameter $B(\Sigma)$, which
measures the maximum number of simulation paths containing any given \ds.

\begin{lem}
\label{lem:count}
Let $\dsq$ be a graphical degree sequence with $d_1\geq d_2\geq \cdots \geq d_n\geq 3$
and let $\Sigma$ be the set of $(\cM,\cM')$-simulation paths
defined in Section~\emph{\ref{sec:irreducible}}.
Then
\[ B(\Sigma)\leq 20 d_1^2\, (2M + d_1^2) \]
where $M = M(\dsq) = \sum_{j\in [n]} d_j$.
\end{lem}

\begin{proof}
Let $e=(X,Y)$ be a transition of the \ds\ chain. We want an upper bound on the
number of simulation paths $\sigma_{GH}$ which contain $e$, where $(G,H)$ is
a transition of the switch chain.  To do this, we will work through the cases
(I)--(IX) from Lemmas~\ref{lem:pivot} and~\ref{lem:no-pivot}.
We will not try to optimise the bounds we determine. We are only interested here in correctness and, to a lesser extent, simplicity.

Let $D = \{a_1,a_2,a_3,a_4\}$ denote the vertices involved in the switch $(G,H)$
that we need to identify.

\begin{itemize}
\item[(I)] In Case~(I) we have $(G,H)=(X,Y)$ which happens in exactly~1 way.
\item[(II)]
In Case~(II), the simulation path has length~2 and only involves the vertices in $D$.
So the simulation path is $G,G_1,H$ and in $G_1$, the edges $a_1a_4 a_2a_3$ are
present and the switch edges $a_1a_2$, $a_1a_3$, $a_2a_4$, $a_3a_4$ are all absent.
Therefore, if $(X,Y)=(G,G_1)$ then we know $G=X$ and $H$ is obtained from $G_1$ by
replacing edges $a_1a_4$, $a_2a_3$ with $a_1a_3$, $a_2a_4$ (the only perfect
matching of $D$ not seen in either $G$ or $G_1$).  This can be done in exactly one way.
By a similar argument, if $(X,Y) = (G_1,H)$ then there is exactly one way to reconstruct
$G$.  So overall, there are 2 possibilities for $(G,H)$ if $(X,Y)$ is involved
in a simulation path in Case~(II).
\item[(III)]
There are exactly 2 possibilities arising from Case~(III), using a very similar argument to that used for Case~(II).
\item[(IV)]
The simulation path defined in Case~(IV) has length~2, so it is $G,G_1,H$ for some $G_1$.
Each \ds\ in the path involves 3 vertices in $D$ and one vertex not in $D$.
(Note that in the second subcase, the vertex $w$ of the triangle on $a_1, u, w$
is not involved in either of the \ds es in the simulation path.)
There are 4 ways to label the vertices involved in a given \ds, and in each case
the remaining vertex of $D$ is a neighbour of one of the four vertices involved in the
\ds.  Hence there are at most $4d_1$ ways to label all vertices in $D$ or
involved in the \ds, and using these vertices we can uniquely construct $H$ from $(G,G_1)$,
or uniquely construct $G$ from $(G_1,H)$.  Hence there are at most $8 d_1$ possibilities
for $(G,H)$ if $(X,Y)$ is involved in a simulation path in Case~(IV).

\item[(V)] The simulation path in Case~(V) has length~4, so it is $G, G_1, G_2, G_3, H$ for some graphs $G_1, G_2, G_3$.
If $(X,Y) = (G,G_1)$ then one of the edges switched out is the diagonal
$a_1a_4$, the the other edge switched out is $uv$ where $u,v\not\in D$.
Then $a_2$ is a common neighbour of of $a_1,u,v$ in $G$, and $a_3$
is a common neighbour of $a_4, u,v$ in $G$. So there are 2 choices for
which of the removed edges is $a_1a_4$ and at most $d_1^2$ ways to choose
the unordered pair
$\{a_2, a_3\}$. From $G=X$ and this information we can recover $H$. (Note that
by symmetry we do not need to multiply by a factor of 2 to decide which vertex of the edge $a_1a_4$ is $a_1$.  The same argument works for the final
\ds\ in the simulation path.

Next, if $(X,Y) = (G_1,G_2)$ then the \ds\ involves the vertices $a_1,a_2,a_3,a_4$, removing the edges $a_1a_2$, $a_3a_4$ which are present in $G$ and
inserting the diagonals $a_1a_4$, $a_2a_3$. Hence we know that
$H[D]$ is given by the only perfect matching of $D$ not seen in either $X$ or
$Y$.  To recover $G$ from $X$ we must know the vertices $u,v$ which are
both common neighbours of $a_2$ and $a_3$, so there are at most $d_1^2$
choices for $u,v$. The same argument holds if $(X,Y)= (G_2,G_3)$.  Overall,
there are at most $6d_1^2$ ways that $(X,Y)$ could be involved in a
simulation path in Case~(V).

\item[(VI)] The simulation path in Case~(VI) has length~4.  Each \ds\ in this
simulation path involves four of the vertices $a_1,a_2,a_3,a_4, u,v,w$.
There are 4 ways to label the vertices involved in each \ds\ in this path,
and then the remaining 3 vertices can be chosen in at most $d_1^3$ ways,
since the subgraph induced by $\{a_1,a_2,a_3,a_4,u,v,w\}$ is
connected at each step.  Once the identities of these vertices are known,
from any \ds\ on the simulation path, we can construct $(G,H)$ uniquely.
It follows that there are at most $16 d_1^3$
ways that $(X,Y)$ could be involved in a simulation path in Case~(VI).

\item[(VII)] The simulation path in Case~(VII) has length~4, consisting
of a triangle-planting \ds, then the simulation path from Case~(IV) in
the resulting graph, then the ``unplanting'' \ds. Lemma~\ref{lem:plant}
is used to choose the triangle-planting \ds, with $v=a_1$
and $R\subseteq \Nb(a_1)\setminus D$. If $(X,Y)$ is the
triangle-planting \ds\ $(G,\widetilde{G})$ then $X=G$.
If (in the application of Lemma~\ref{lem:plant}) $\ell=4$ then the
triangle-planting \ds\ involves one vertex
from $D$ and three vertices outside $D$.  There are 4 ways to label the
vertices involved in the \ds, which identifies $a_1$ (say).
Then there at most $d_1$ ways to choose $a_2$ from the neighbours of $a_1$,
and at most $M$ ways to choose the vertices $(a_3,a_4)$. Using
these vertices we can construct $H$. So there are at most $4d_1 M$
ways that $(X,Y)$ can be the triangle-planting \ds\ when $\ell=4$,
and the same bound holds when $\ell=5$.

However, if $\ell\geq 6$ then
the triangle-planting \ds\ does not involve $a_1$, but involves two
neighbours $u_1, u_2$ of $a_1$ (the neighbours form a triangle with $a_1$
after the \ds).  Here there are 2 choices for the edge $u_1u_2$
which is inserted by the \ds, then at most $d_1^2$ choices for $(a_1,a_2)$
and at most $M$ choices for $(a_3,a_4)$.  Once these vertices are known,
we can construct $H$ from $G$.  There are at most $2d_1^2 M$ ways
that $(X,Y)$ can be the triangle-planting switch when $\ell\geq 6$.
Hence, overall there are at most $8d_1 M + 2 d_1^2 M$ ways that
$(X,Y)$ could be the triangle-planting \ds\ in Case~(VII).
The same bound holds for the ``unplanting'' \ds\ at the end of the simulation path, and we must also add the $8d_1$ choices that arise if $(X,Y)$ is used as the second or third \ds\ in the simulation path. (We multiply $8d_1$ by three as there
are three possibilities for $\ell$, as described above.) As a grand total,
there are at most
\[ 4d_1^2 M + 16 d_1 M + 24d_1\]
possiblities for $(G,H)$ if $(X,Y)$ is part of a simulation path
in Case~(VII).

\item[(VIII)] In subcase (VIIIa), the triangle-planting \ds\ involves
two vertices from $D$, say $a_1, a_3$, and two vertices outside $D$.
There are 4 ways to label the vertices involved in the \ds\ and then
at most $d_1^2$ ways to choose $(a_2,a_4)$.  So if $(X,Y)$ is the
triangle-planting \ds\ in subcase (VIIIa) then there are at most
$4d_1^2$ options for $(G,H)$, as $G=X$ and $H$ can be constructed
from $G$ using these known vertices.  The same bound holds for the
unplanting \ds\ in subcase (VIIIa), and again we have at most $8d_1$
options from the second or third \ds\ in the simulation path, giving
at most $8d_1^2 + 8d_1$ choices for $(G,H)$ from subcase (VIIIa).

In subcase (VIIIb), the triangle-planting \ds\ involves exactly
one vertex from $D$, say $a_1$, and three vertices outside $D$.
One of the vertices involved in the triangle-planting \ds\ is a neighbour
of $a_3$, so there are at most $4d_1^3$ ways to identify
all vertices in the triangle-planting \ds\ and all vertices in $D$.
The same bound holds for the unplanting \ds, and again we must include
the bound $8d_1$ to cover the second and third \ds es in the
simulation path.  Hence there are at most $8d_1^3 + 8d_1$ choices
for $(G,H)$ if $(X,Y)$ is part of a simulation path in subcase (VIIIb).

Overall, there are at most
\[ 8d_1^3 + 8d_1^2 + 16d_1\]\
choices for $(G,H)$ if $(X,Y)$ is part of a simulation path in
case~(VIII).

\item[(IX)] In subcase (IXa), the simulation path has length~3, so it is
$G, G_1, G_2, H$ for some graphs $G_1, G_2$.  (Here we do not plant a triangle.)
The first \ds\ involves one vertex of $D$, say $a_3$, and a neighbour of $a_1$.
So there are at most $4d_1^3$ ways to identify all vertices in the \ds\
as well as $(a_1, a_2, a_4)$.  Here $G=X$ and we can construct $H$ using the
vertices already identified.  The second \ds\ $(G_1,G_2)$ involves the
four vertices of $D$, with $u_1$ a common neighbour of $a_1$ and $a_3$ (say) and
where $u_1w_1w_2v$ is a path in $G_1$.  Hence there are at most $2d_1^4$ ways to
identify each of the vertices of $D\cup \{u_1,w_1,w_2\}$.  Using this we can
construct $G$ from $G_1$ and then construct $H$.  Finally, the third switch involves
exactly one vertex from $D$, say $a_3$, but now it is possible that $a_2$ and $a_4$
belong to a different component than $a_3$ in both $G_2$ and $H$.
Since $H=Y$, there are at
most $4d_1 M$ ways to identify the vertices of $D$ and then construct $G$ from $H$.
Overall there are most $4d_1 M + 2 d_1^4 + 4d_1^3$ choices for $(G,H)$ if
$(X,Y)$ is part of a simulation path in subcase (IXa).

In subcase (IXb), we call Lemma~\ref{lem:plant} to identify a triangle-planting
switch $(G,\widetilde{G})$.  Then from $\widetilde{G}$ we perform two \ds es as described in Case~(IV), then we
``unplant'' the triangle by performing the inverse \ds\ in the resulting
graph.  There are several possibilities which we must analyse:
\begin{itemize}
\item If Lemma~\ref{lem:plant} uses a 4-cycle which includes the path $a_1u_1w_j$
for $j\in\{1,2\}$, then the triangle-planting \ds\ does not involve any
	vertex of $D$,
	but involves a neighbour of $a_1$, say.  Hence there
	are at most $4 d_1^2 M$ ways to identify the vertices of $D$.
	Since $G=X$, this allows us to construct $H$.
\item If $a_1u_1 w_j$ is not part of a 4-cycle or 5-cycle ($j=1,2$)
then the triangle-planting
\ds\ involves $a_1$, and hence there are at most $4d_1 M$ ways to identify
the vertices of $D$ and hence construct $H$ from $G=X$.
\item If $a_1u_1 w_j$ is in a 5-cycle 
but is not in any 4-cycle,
then
the triangle-planting \ds\ does not involve any element of $D$ but it involves
neighbours of $a_1$, say.  So there are at most $4d_1^2 M$ ways to identify the
vertices of the \ds\ and of $D$, which allows us to construct $H$ uniquely from $G=X$.
\end{itemize}
In each case, the same bounds hold for the final ``unplanting'' \ds, and there are
at most $8d_1$ choices for $(G,H)$ which arise if $(X,Y)$ is the second or third
\ds\ in the simulation path.  (We multiply $8d_1$ by 3 to account for the three
possible situations for $\ell$.)
In total, this gives at most
\[ 2(8 d_1^2 M + 4 d_1 M ) + 24d_1 = 16d_1^2 M + 8d_1 M + 24d_1\]
choices for $(G,H)$ when $(X,Y)$ is part of
a simulation path in subcase (IXb).


In subcase (IXc), we require a \ds\ to plant a triangle, say $(G, \widetilde{G})$.
The \ds\ involves no vertex from $D$ but involves a neighbour of one vertex
from $D$, say $a_1$.  Here $G=X$ and there are at most $4 d_1^2 M$ ways to identify
the elements of the \ds\ vertices and $D$, and hence construct $H$ from $G$.
The same bound holds for the final ``unplanting'' \ds.  The other three
\ds es in the simulation path match those in subcase (IXa), and we proved
above that there are at most
$4d_1 M + 2 d_1^4 + 4d_1^3$ choices for $(G,H)$ if $(X,Y)$ is one of these
three \ds es.
This gives at most
\[ 4d_1 M + 2 d_1^4 + 4d_1^3 +2(4d_1^2M) = 8 d_1^2 M + 4d_1 M + 2 d_1^4 + 4d_1^3\]
choices for $(G,H)$ when $(X,Y)$ is part of a simulation path in subcase (IXc).

Hence overall, there are at most
\[ 24 d_1^2 M + 16 d_1 M + 4 d_1^4 + 8d_1^3 + 24d_1\]
choices for $(G,H)$ if $(X,Y)$ is part of a simulation path in Case~(IX).
\end{itemize}

Adding up contributions from every case, we see that a \ds\ $(X,Y)$
can belong to the simulation path of at most
\[  28 d_1^2 M + 32 d_1 M + 4 d_1^4 + 32 d_1^3 + 14d_1^2 + 72d_1 + 5\]
pairs $(G,G')$.
We can simplify this using the assumption that $d_1\geq 3$, and the fact that
$M\geq 2d_1$ (as can be seen by considering a star $K_{1,d_1}$).
This gives an upper bound of
\[ \left(28 + \frac{32}{3}\right) d_1^2 M +
   \left(4 + \frac{32}{3} + \frac{14}{9} + \frac{72}{27} + \frac{5}{81}\right) d_1^4
   < 20d_1^2\, (2M + d_1^2).
\]
   This completes the proof.
\end{proof}

We can now prove Theorem~\ref{thm:rapid}. 

\begin{proof}[Proof of Theorem~\ref{thm:rapid}]
Suppose that the switch chain is rapidly mixing on $\Gnd$ for all 
sequences $\dsq\in\mathcal{D}$. 
By Theorem~\ref{thm:Metropolis-switch}, we can assume that
the same is true for
the modified Metropolis switch chain $\cM'$. That is, there exists
	a polynomial $p(n)$ such that if $\dsq\in\mathcal{D}$ and $\dsq$ has
	length $n$ then the mixing time of $\cM'$ on $\Gnd$ is at most $p(n)$.
Using results of Sinclair~\cite{sinclair},
Guruswami~\cite[Theorem~4.9]{guruswami} proved that there exists a set $\Gamma'$
of canonical paths for $\cM'$ and a polynomial $q(n)$ such that
$\rho(\Gamma')\leq q(n)$ for all $\dsq\in\mathcal{D}$ with length $n$.

Now we calculate the other parameters needed to apply Theorem~\ref{thm:CDGH}. 
By Theorem~\ref{thm:five}, 
the maximum simulation path length  is $\ell(\Sigma) = 5$.
Both chains have the same 
stationary distribution, namely the distribution which assigns to 
$G\in \Gnd$ the
probability $\lambda^{\min\{t(G),\nu\}}/\hat{Z}_\lambda(\dsq)$, where
$\hat{Z}_\lambda(\dsq)$ is the normalising factor.  
Hence the stationary ratio is given by $R(\cM,\cM')=1$.
Similarly, the transition probability $P(G,H)$ is given by 
\[ P(G,H) = \frac{1}{3\, a(\dsq)}\, \min\{\lambda^{\min\{t(H),\nu\}-\min\{t(G),\nu\}},1\}\]
whenever $G\neq H$ and $(G,H)$ is a switch (for $\cM'$) or a \ds\ (for $\cM$).
Therefore the simulation gap satisfies
\[ D(\cM,\cM') = \max_{\substack{uv\in E(\cM)\\zw\in E(\cM')}}
  \frac{\min\{ \lambda^{\min\{ t(z),\nu\}},\, \lambda^{\min\{t(w),\nu\}}\}}
        {\min\{ \lambda^{\min\{ t(u),\nu\}},\, \lambda^{\min\{t(v),\nu\}}\}}
	\leq \lambda^{\nu} \leq n^\alpha,
\]
using the assumptions of the theorem and (\ref{eq:upper-lambda-nu}) for the final inequality.
Hence it follows from Theorem~\ref{thm:CDGH} that there exists a set $\Gamma$
of canonical paths for $\cM$ with congestion
\[ \bar{\rho}(\cM) \leq D(\cM,\cM')\, R(\cM,\cM') \ell(\Sigma)\, B(\Sigma) \,\bar{\rho}(\Gamma') \leq 100n^\alpha\,  d_1^2\, (2M+d_1^2)\, q(n).\]
Arguing as in the proof of Theorem~\ref{thm:Metropolis-switch} we can conclude
that the mixing time $\tau_\lambda(\varepsilon)$ of the modified \ds\ chain is 
bounded above by
\[ \tau_\lambda(\varepsilon) \leq  100n^{\alpha} \, d_1^2\, (2M + d_1^2)\, q(n)\,
    \big( M\log M + \alpha \log n + \log(\varepsilon^{-1})\big).
    \]
Thus, the modified Metropolis \ds\ chain with parameter~$\lambda$ is rapidly mixing
for all $\dsq\in\mathcal{D}$, under the assumption that $\lambda\mu\leq \log^\alpha n$ for some $\alpha\in (0,1)$.
\end{proof}

The assumption that $\lambda\mu\leq \log^\alpha n$ is fairly restrictive,
but allows any constant value of $\lambda$ when $d_1=O(1)$.  
To go beyond this assumption would require a different approach to 
analysing the Markov chain.

\subsection{Slow mixing} \label{ss:slowmix}

The upper bound on $\lambda\mu$ required for Theorem~\ref{thm:rapid}
is largely a consequence of the Metropolisation. Consider a $d$-regular graph $G$ which consists of $k$ disjoint copies of $K_{d+1}$. (Note that these graphs are central to the irreducibility proof in~\cite{CDG21}.) Then $G$ contains no paths of length 4, so no \dsp\ is possible. Any \dsm\ must take one edge from each of two disjoint copies of $K_{d+1}$, and hence will destroy $2(d-1)$ triangles. So any change to $G$ will occur only with probability $\lambda^{-2(d-1)}$. Suppose that $\lambda>1$ is any constant and $d=\gamma\log n$, where $\gamma\to\infty$ as $n\to\infty$. Then, with high probability, no change will occur within $(1/\lambda)^d=n^{\gamma\log\lambda}$ steps of the chain, which is superpolynomial.

There are exponentially few $d$-regular graphs of this form, as a fraction
of the set of all $d$-regular graphs on $n=kd$ vertices, 
but each such graph gives a potential bottleneck for mixing. The small proportion suggests a solution by ``ignoring'' these graphs in some way. Unfortunately, the difficulty is not restricted to these graphs. Suppose we have only a single $K_{d+1}$ component. In order to break the $K_{d+1}$, we require the two vertices of a non-edge in the remainder of the graph to have many common neighbours, at least $d/4$, say. Then a \dsm\ could create at most $d/2$ triangles, which is not enough to offset the $d-1$ triangles that have been destroyed. If there are $o(n)$ triangles in the current graph, then we can bound the probability of such a non-edge appearing in the graph by
$O(n^{d/4+2}(d/n)^{d/2})=o(n^{-d/5})$, which is exponentially small if $d\to\infty$ with $n$. Thus all \dsm's involving the $K_{d+1}$ will have exponentially small probability of acceptance, and hence mixing will take exponential time.

Of course, the proportion of $d$-regular graphs with $n$ vertices and at least one $K_{d+1}$ component is still small. But, if $t(G)=\Omega(n)$, it is shown in~\cite{HLM} that $G$ contains many dense clique-like structures, called pseudocliques, which would require many steps to disassemble. Therefore, it seems unlikely that the \ds\ chain could generate regular graphs with $\Omega(n)$ triangles in polynomial time, except possibly for small constant~$d$. Nevertheless, the modified \ds\ chain may well be rapidly mixing for a wider range of $\mu$ and $\lambda$ than we are able to prove here.


\section{The asymptotic distribution of $t(G)$}\label{sec:distribution}

The aim of this section is to analyse the distribution of the number of
triangles in a random graph in $\Gnd$, both under the uniform distribution
and the Gibbs distribution $\pi_\lambda$.  We do this to justify the use
of the cut-off $\nu$ in the modified Metropolis switch chain and the
modified Metropolis \ds\ chain, see Remark~\ref{rem:indistinguishable}, though
the results of this section may also be of independent interest.

We denote the Poisson distribution with mean $\mugeneral$ by \pois{\mugeneral} and, if the random
variable $X$ has this distribution, we write in standard notation $X\distas\pois{\mugeneral}$ or $\law{X} =\pois{\mugeneral}$.
We wish to determine the approximate distribution of $t(G)$ for large $n$, 
under the uniform distribution on $\Gnd$. We first give an appropriate
definition of distributional approximation in this context, which has wider applicability.  We note that the traditional approach to distributional approximation, particularly in Statistics, has been through limit distributions,
usually with little or no error estimation.

\subsection{Distributional approximation}\label{sec:asymptotic}

Let $\cX_n$ be a sequence of discrete state spaces, and let $X_n,X'_n$ be random variables on $\cX_n$
such that $\law{X_n}=p_n,\law{X'_n}=p'_n$. We will say that $X_n$ \emph{approximates $X'_n$ asymptotically}, which we denote by $X'_n\distapp X_n$, 
if $p'_n(A)=p_n(A)+o(1)$ for all $A\subseteq\Sigma$ as $n\to\infty$,
That is, $X'_n\distapp X_n$ if and only if
\[ \max_{A\subseteq \cX}\big|\Pr(X'_n\in A) - \Pr(X_n\in A) \big| = o(1).\]
This involves a known metric on distributions,
the total variation distance, $\dtv(p'_n,p_n)$, as defined in~\eqref{eq:dtv}. So $X'_n\distapp X_n$ (and vice versa) if the total variation distance
\[ \dtv(p'_n,p_n) = \tfrac{1}{2}\sum_{x\in\cX_n} |p'_n(x) - p_n(x)|=o(1)\mbox{\,\ as\ \,}n\to\infty.\]
This relation is symmetric, reflexive and transitive, since $\dtv$ is a metric. However, transitivity holds only over a finite number of applications, since we must preserve the $o(1)$ term. So this is only a local equivalence relation on sequences $(\cX_n,p_n)$.

If $X_n$ has a known distribution, $\pois{\mugeneral}$ say, we also write $X'_n\distapp \pois{\mugeneral}$. This notation for distributional approximation is taken from~\cite{reinert}, but was used only informally there.

This definition is not new. Janson~\cite{Janson} used the term ``asymptotic equivalence'' for this relationship between sequences of random variables, using the notation $(X_n)\cong (X'_n)$. However, the focus in~\cite{Janson} is on the case where both $X'_n$ and $X_n$ are unknown distributions, and the equivalence indicates a stronger relationship than contiguity. There the relationship between $X_n$ and $X'_n$ is symmetrical. Here we assume $X_n$ has a known distribution which approximates the unknown distribution of $X'_n$, so the relationship is not symmetrical. Also, as we have noted above, this is not a true equivalence relation. Therefore, we prefer the above terminology and notation.

Note that requiring only $p'_n(x)=p_n(x)+o(1)$ for all $x\in\cX_n$ is not a useful notion of approximation. For example, suppose $\cX_n=[2n]$, $p_n(x)=1/n$ ($x\in[n]$), $p'_n(x)=1/n$,
($x\in[2n]\sm[n]$). Then $X_n$ approximates $X'_n$ in this weaker sense, since $p'_n(x)=p_n(x)\pm 1/n$.
However, $\dtv(p_n,p'_n)=1$ for all $n$, so this does not seem a reasonable approximation.

Though we will only use it in the discrete case, the approach to distributional approximation outlined above can be extended to arbitrary probability measures, as is done in~\cite{Janson}. It was initially developed for normal approximations~\cite{stein}, but later for other distributions (e.g.~\cite{chen,PRS,reinert}). The methods used have been elementary (e.g.~\cite{serfling}), coupling arguments (e.g.~\cite{PRS}), Stein's method~\cite{stein,BHJ} and semigroup methods (e.g.~\cite{LeCam,steele}). The aim has been to quantify errors in limit distributions in Statistics, or to quantify rates of convergence to these limit distributions, rather than to quantify asymptotic approximations.

We will use only elementary methods, but with an alternative criterion, which we now show is equivalent to the definition above.
\begin{lem}\label{lem:asympt}
  $X_n\distapp X'_n$ if and only if there is a set $S_n\subset\cX_n$ such that, as $n\to\infty$,
  $p'_n(x)=(1+o(1))p_n(x)$ for all $x\notin S_n$ and $\max\{p_n'(S_n),p_n(S_n)\}=o(1)$.
\end{lem}
\begin{proof}
  If $S_n$ exists then there exist sequences $0<\delta_n,\epsilon_n=o(1)$ so that $|p'_n(x)-p_n(x)|<\delta_np_n(x)$ for $x\notin S_n$ and $\max\{p_n'(S_n),p_n(S_n)\}<\epsilon_n$. Hence
  \[ \dtv(p_n,p'_n) < \epsilon_n + \delta_n\sum_{x\notin S_n}p_n(x) < \epsilon_n + \delta_n =o(1),\]
  so $X_n\distapp X'_n$.

  Conversely, if $X_n\distapp X'_n$ then there is a sequence $0<\eta_n=o(1)$ such that $\dtv(p_n,p'_n)<\eta_n$.
  Let $\epsilon_n=\sqrt{\eta_n}=o(1)$ and define
  \[ S_n=\left\{x\in\cX_n:\frac{|p'_n(x)-p_n(x)|}{\max\{p_n(x),p'_n(x)\}}\geq \epsilon_n\right\}.\]
  Then
  \[ \epsilon_n^2 > \dtv(p_n,p'_n) \geq\epsilon_n\sum_{x\in S_n}\max\{p_n(x),p'_n(x)\}\geq \epsilon_n\max\{p_n(S_n),p'_n(S_n)\},\]
  so $\max\{p_n(S_n),p'_n(S_n)\}< \epsilon_n=o(1)$.

  If $x\notin S_n$, then we have $|p'_n(x)-p_n(x)|< \epsilon_n\max\{p_n(x),p'_n(x)\}$.
  By taking $n$ sufficiently large, we can assume that $\epsilon_n < \nicefrac12$, and hence
  $p'_n(x)\leq 2 p_n(x)$ for all $x\not\in S_n$.
  For all such $n$ and all $x\not\in S_n$, we conclude that
  $p'_n(x)/p_n(x) \in (1-2\epsilon_n,1+2\epsilon_n)$,
  and hence $p'_n(x) \in (1\pm\delta_n)p_n$
  where $\delta_n=2\epsilon_n=o(1)$ as $n\to\infty$.  This completes the proof.
\end{proof}

This formulation also gives explicit bounds on errors. In one direction, we can take $\eta_n=\epsilon_n + \delta_n$, while in the other we can take $\epsilon_n=\sqrt{\eta_n}$ and $\delta_n=2\sqrt{\eta_n}$.

We will apply Lemma~\ref{lem:asympt} to prove the following result.

\begin{thm}
Suppose that $\dsq$ is a degree sequence with $d_1=o(n^{1/9})$ and $d_n\geq 1$, such that $M_2\geq M$.  Let $\cX_n=\{0,1,\ldots,M_2/6\}$. 
\begin{enumerate}
\item[\emph{(i)}] Let $X_n'$
be the number of triangles $t(G)$ when $G$ is chosen uniformly at random from 
$\Gnd$.
Then $X_n'\distapp \emph{\pois{\mu}}$ where $\mu=\mu(\dsq)$.
\item[\emph{(ii)}]
Further suppose that $\lambda>1$ is a constant with $d_1 \log \lambda = o(\log n)$, and let $X_n'$
be the number of triangles $t(G)$ when $G$ is chosen from $\Gnd$ using the
Gibbs distribution $\pi_\lambda$.
Then $X_n'\distapp \emph{\pois{\lambda \mu}}$ where $\mu=\mu(\dsq)$.
\end{enumerate}
\label{thm:asympt-Poisson}
\end{thm}

In particular, part (ii) implies that the mean number of triangles under $\pi_\lambda$ is $\lambda\mu$, compared with $\mu$ under the uniform distribution.
Hence, in order to increase the number of triangles as far as possible, we should take $\lambda$ as large as possible, under the constraint $d_1\log\lambda=o(\log n)$.

To prove Theorem~\ref{thm:asympt-Poisson} we will take $S_n=\{G:t(G) > t_0\}$, for a suitably large $t_0$, to be determined.  
Thus $S_n$ is the tail of the distribution of $t(G)$. 
We divide the proof
by analysing first ``small'' values of $t<t_0$ in Section~\ref{sec:small_t}
and then ``large'' values of $t>t_0$ in Section~\ref{sec:small_t}.
Thus, in Section~\ref{sec:small_t} we consider $G\notin S_n$, and in Section~\ref{sec:large_t} we consider $G\in S_n$.

\medskip

Before proceeding, we make some more remarks about distributional approximations.
Note that $X'_n\distapp X_n$ does not imply that $X'_n,X_n$ have exactly the same properties.
Suppose $Y_n=h_n(X_n), Y'_n=h_n(X'_n)$ for a function $h_n:\cX_n\to\mathbb{C}$ (the complex numbers).
Let $\norm{h_n}=\max_{x\in\cX_n}|h_n(x)|$. Then we can only claim $\Expn(Y'_n)\sim\Expn(Y_n)$
if $h_n$ is uniformly bounded in $n$ using this norm. This follows from the (easily proved) inequality
\[ \big|\Expn[Y'_n]-\Expn[Y_n]\big| \leq \norm{h_n}\dtv(p_n,p'_n),\]
with equality possible if $h_n(x)=\sign\big(p'_n(x)-p_n(x)\big)$, where $\sign(x)\in\{-1,0,1\}$ (in the obvious way),
so $\norm{h_n}=1$ (unless $p'_n=p_n$ and $\dtv(p'_n,p_n)=0$). Indeed, this equality is another characterisation of $\dtv$. Note that the uniform boundedness requirement includes $X_n$ itself, so care is needed.

As an example, suppose $\cX_n=[n]$, $p_n=C_n/i^2$ ($i\in[n]$),
where $C_n\sim\pi^2/6$. (This example is related to the Cauchy distribution.)
However, $\Expn[X_n]=\Omega(\log n)\to\infty$ as $n\to\infty$.
The explanation is that $n$ is not a uniformly bounded function of $n$.
This is a known pathology in the limit theory of distributions. In this example,
$X'_n$ converges pointwise to a fixed distribution, but
$\lim_{n\to\infty} \Expn[X'_n]$ does not even exist.

In the limit distribution approach to approximation, we must restrict attention to uniformly
bounded functions. Here, if $h_n$ is not uniformly bounded, we have the additional option of estimating the contribution to the expectation from $S_n$, as we do below for Poisson approximations.

However, if $X'_n\distapp X_n$ for real-valued random variables, then $X_n$ and $X'_n$ do have the same
characteristic function $\Expn(e^{itx})$ in the limit as $n\to\infty$,
since $|e^{itx}|=1$ for all real $t$ and $x$. This could be used,
for example, to prove asymptotic normality of the Poisson approximations below in the limit
distribution sense.

Note that we cannot use this idea of approximation to compare discrete distributions
with continuous distributions. Discrete points have measure zero in a continuous distribution, so
we always have $\dtv(X_n,X'_n)=1$. Thus, any definition of approximation in this context must
use a weaker metric than $\dtv$. For real-valued random variables, the usual choice is the \emph{Kolmogorov metric} $\max_{x\in\bbR}|F_n(x)-F'_n(x)|$, where $F_n,F'_n$ are the respective distribution functions of $X_n,X'_n$.
Thus, the asymptotic normal results for triangles proved in~\cite{Gao} are valid for larger values of $\mu(\dsq)$
than the asymptotic Poisson results proved here, but they are much weaker approximations. In fact, the indirect method of proof used in~\cite{Gao} does not seem to allow for any quantification of errors.

\subsection{Small number of triangles}\label{sec:small_t}

In this section we will prove the following result.

\begin{thm}\label{prob:thm10}
Let $G$ be chosen uniformly from $\Gnd$ and let $\mu=\mu(\dsq)$.
Suppose that $M_2\geq M$ and that $t_0=\min\{\sqrt{M/d_1^3},\omega(n)d_1^3\}$ for some function $\omega(n)\to\infty$ as $n\to\infty$.
Then $\Pr(t(G)=t)\sim e^{-\mu}\mu^t/t!$ for all $t=o(t_0)$. 
\end{thm}

Our proof will be based on the \emph{switchings} technique of McKay~\cite{McKay1985} and others,
combined with the conditional probability estimates of Gao and Ohapkin~\cite{GaoOha}. 
The switchings technique has been applied in the analysis of the switch chain before, see for example~\cite{GG,GS}.
For reference, we paraphrase \cite[Corollary 2]{GaoOha} as follows. 

\begin{lem} \emph{(\cite[Corollary~2]{GaoOha})}\
\label{prob:eq10}
Let $H^+\in\Gnd[n,\dsq']$, where $\dsq'<\dsq$,
and let $M'=\sum_{i=1}^n d'_i=2|E(H^+)|$. 
Then, if $G$ is chosen uniformly from $\Gnd$,
\[
\Pr(uv\in G\mid H^+\subset G) = \bigg(1+O\Big(\frac{d_1^2}{M-M'}\Big)\bigg)\frac{(d_u-d'_u)(d_v-d'_u)}{M-M'+(d_u-d'_u)(d_v-d'_u)}.
\]
\end{lem}

In our application, $H^+$ will be a subgraph $H_t$ of $G$ containing exactly $t$ triangles with $t=o(M)$, which will imply that the error term in Lemma~\ref{prob:eq10} is $O(d_1^3/M)$.

\begin{rem}
In~\cite{GaoOha} there is also a subgraph $H^-$, edge-disjoint from $H^+$, which is conditioned on being absent from $G$. Since we make no use of this, we omit it. Also, we make the $1+o(1)$ term explicit in Lemma~\ref{prob:eq10}, since we iterate these estimates and hence exponentiate this term. We start with $H^+=\es$, $\dsq'=\mathbf{0}$, so $H^+$ remains small relative to~$G$.
\end{rem}

\begin{rem}
Note that Lemma~\ref{prob:eq10} gives essentially the unconditioned estimate on the graph which results from deleting the edges of $H^+$ from $G$, as might be expected using the configuration model for graphs with given degree sequences.
\end{rem}

We apply these estimates with the switching method, as follows. 
For $t=0,\ldots, M_2/6$, let $\cN_t=\{ G\in\Gnd: t(G)=t\}$
be the set of graphs with exactly $t$ triangles, and let $N_t=|\cN_t|$.
We will use $G_t$ to denote an element of $\cN_t$ chosen uniformly at random.
Consider the set $S_t^+$ of \dsp es from graphs in $\cN_t$ which create exactly one triangle and destroy none. Also define the set of the inverses of these \dsp es, namely the set $S_{t+1}^-$ of \dsm es from graphs in $\cN_{t+1}$ which delete exactly one triangle and create none. Let $s_t^+=|S_t^+|$ and $s_t^-=|S_t^-|$. Then we have
\begin{equation}\label{prob:eq20}
  N_t\, \Expn[s_t^+\mid G_t]\,=\,N_{t+1}\,\Expn[s_{t+1}^-\mid G_{t+1}],
\end{equation}
from which we can estimate $N_{t+1}/N_t$. We note that McKay and co-authors have considerably generalised this switchings idea,
for example in~\cite{HashMc}, but here we only require the simple version given above.

Recall that $\mu=\mu(\dsq)=\frac{M_2^3}{6M^3}$. We first prove a claim made in Section~\ref{sec:notation} above.
\begin{lem}\label{prob:lem05}
Let $G$ be chosen uniformly from $\Gnd$.
  If $d_1^2=o(M)$ and $M_2\geq M$ then $\Expn[t(G)]\sim\mu.$
\end{lem}
\begin{proof}
Using~Lemma~\ref{prob:eq10}, and calculations similar to cases~\ref{case_a} and \ref{case_b} of Lemma~\ref{prob:lem10} below, we have, for triangles $i,j,k$,
\begin{equation}\label{prob:eq25}
\Expn[t(G)] = \Big(1+O\Big(\frac{d_1^2}{M}\Big)\Big)^3\, \sum_{\substack{i,j,k=1\\\text{distinct}}}^n\, \frac{d_i(d_i-1)d_j(d_j-1)d_k(d_k-1)}{6M^3}
 \sim \frac{M_2^3}{6M^3} = \mu,
\end{equation}
where the divisor 6  counts permutations of $i,j,k$ representing the same triangle. The exponent 3  on the error term in \eqref{prob:eq25} results from three uses of Lemma~\ref{prob:eq10} and replacing the three occurrences of $M+O(d_1^2)$ with $M$. A relative error of $O(d_1^2/M_2)$ arises from the
terms where $i,j,k$ are not distinct: these terms arise in $M_2^3$ but are
not present in the original sum.  Therefore we need to assume $M_2 = \Omega(M)$,
but we make the stronger assumption $M_2\geq M$. 
\end{proof}
We have $\mu(\dsq)=\frac{M_2^3}{6M^3}\geq\nicefrac16$ under our assumption that $M_2\geq M$, so $\mu=\Omega(1)$ as $n\to\infty$.

Note that the numerator in the summation in \eqref{prob:eq25} is independent of the
order in which the edges of the triangle are exposed, and depends only on its degree structure.
This observation is true for all small subgraphs, and we use it below without further comment.
\begin{lem}\label{prob:lem10}
  If $td_1^3=o(M)$ and $M_2\geq M$ then $\Expn[s_t^+\mid G_t]\,=\, \big(1+O(td_1^3/M)\big)\, 6M\mu$.
\end{lem}
\begin{proof}
We first give an estimate of $\Expn[s_t^+\mid G_t]$, and then show that it is asymptotically tight.
A \dsp\ corresponds to a path of length 5 in $G$, which we will denote by $P_5$, though we make
no assumptions about other possible edges between its vertices.
The \dsm\ corresponding to this $P_5$ starts with a triangle on its 3 middle edges, and a disjoint edge between its end vertices.
We will label these as illustrated below.
\medskip

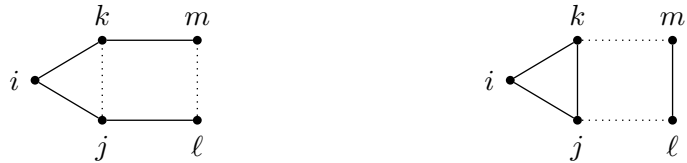
\centerline{%
\begin{tikzpicture}[yscale=0.75,xscale=1.25,font=\small]
  \begin{scope}
  \node[b,label=left:$i$]  (i)                      {};
  \node[b,label=below:$j$] (j) at ($(i) + (-45:1)$) {};
  \node[b,label=above:$k$] (k) at ($(i) + ( 45:1)$) {};
  \node[b,label=below:$\ell$] (l) at ($(j) + ( 0:1)$) {};
  \node[b,label=above:$m$] (m) at ($(k) + ( 0:1)$) {};
  \draw (l)--(j)--(i)--(k)--(m) ;
  \draw[dotted] (j)--(k) (l)--(m) ;
  \end{scope}
  \begin{scope}[xshift=5cm]
  \node[b,label=left:$i$]  (i)                      {};
  \node[b,label=below:$j$] (j) at ($(i) + (-45:1)$) {};
  \node[b,label=above:$k$] (k) at ($(i) + ( 45:1)$) {};
  \node[b,label=below:$\ell$] (l) at ($(j) + (  0:1)$) {};
  \node[b,label=above:$m$] (m) at ($(k) + (  0:1)$) {};
  \draw (k)--(j)--(i)--(k) (l)--(m) ;
  \end{scope}
  \draw[dotted] (j)--(l) (k)--(m) ;
\end{tikzpicture}}\vspace{1ex}

Then our initial estimate is the number of $P_5$'s, where we will temporarily ignore other restrictions:
\begin{align}
\Expn[s_t^+\mid G_t] &\leq
 \big(1+O(d_1^2/M)\big)^4\, \sum_{\substack{i,j,k,\ell,m=1\\\text{distinct}}}^n\frac{d_i(d_i-1)\, d_j(d_j-1)\, d_k(d_k-1)\, d_\ell\, d_m}{M^4}\nonumber \\
 &\leq \big(1 + O(d_1^2/M)\big)\, \frac{M_2^3M^2}{M^4} = 6M\mu,
\label{prob:eq30}
\end{align}
where the expression in the summation is obtained using Lemma~\ref{prob:eq10} with $H_t=\es$.

The estimate~\eqref{prob:eq30} is an upper bound because we have ignored conditions that should be
imposed and which reduce the expectation. We will now refine this by considering each possible overestimate in turn.
This will give us a lower bound on the expectation which matches the upper bound asymptotically for small enough $t$ and $d_1$.
\begin{enumerate}[topsep=0pt,itemsep=0pt,label=(\alph*)]
  \item\label{case_a} First, we should have $i\neq j\neq k\neq \ell\neq m$. There are $\binom{5}{2}=10$ possible first-order equalities in \eqref{prob:eq30},   and we claim that each of these gives rise to a $1+O(d_1^2/M)$ error factor. Consider, for example if $i=j$. If the other vertices are distinct, then \eqref{prob:eq30} becomes
  \begin{equation*}
  \sum_{i,k,\ell,m=1}^n\frac{d_i^2(d_i-1)^2d_k(d_k-1)d_{\ell}d_m}{M^4}\leq \frac{d_1^2 M_2^2M^2}{M^4} = \frac{d_1^2M\mu}{M_2}\leq d_1^2\mu.
  \end{equation*}
  Comparing this with \eqref{prob:eq30} gives the claim. So overall the effect is no worse than a $(1+O(d_1^2/M))^{10}=1+O(d_1^2/M)$ factor. Note that higher-order equalities are automatically included, in fact over-counted, in these calculations.
  \item\label{case_b} We have used $M$ in the denominator in~\eqref{prob:eq30} rather than the $M-M'$ from Lemma~\ref{prob:eq10}. In fact, $H_t$ is a set of $t$ triangles, so $M'\leq 6t$, so the denominator in Lemma~\ref{prob:eq10} is at most $M+d_1^2$, and at least $M-6t$. So using $M$ produces a $(1+ O(t+d_1^2)/M)^4=1+ O(t+d_1^2)/M$ error factor in~\eqref{prob:eq30}.
  \item\label{case_c} We have not considered the effect of $H_t$ on the
  calculation of the quantity $M'_2=\sum_{i=1}^n (d_i-d'_i)(d_i-d'_i-1)$
  which corresponds to $M_2$ on $G$. We have
   \begin{equation*}
    M_2 \geq M'_2=\sum_{i=1}^n (d_i-d'_i)(d_i-d'_i-1)\geq \sum_{i=1}^n d_i(d_i-1)-2d_1\sum_{i=1}^n d'_i \geq M_2-12td_1.
    \end{equation*}
    So using $M_2$ produces a $(1+ O(td_1)/M_2)^3=1+ O(td_1)/M$ error factor in~\eqref{prob:eq30}.
  \item\label{case_d} We wish to ensure that almost all $P_5$'s do not contain an edge of $H_t$. We claim that these give rise to a $1+O(td_1^2/M)$ error factor. Consider, for example if $ij\in E(H_t)$. Then \eqref{prob:eq30} becomes
\begin{align*}
  2\sum_{ij\in E(H_t)}\sum_{k,\ell,m=1}^n\frac{(d_i-1)(d_j-1)d_k(d_k-1)d_{\ell}d_m}{M^3}&\leq\frac{2td_1^2 M_2M^2}{M^3}
   = \frac{2td_1^2 M_2^3M^2}{M^3M_2^2}\\ &= \frac{12td_1^2 \mu M^2}{M_2^2} = O(td_1^2\mu).
  \end{align*}
  Again, comparing this with \eqref{prob:eq30} gives the claim. Since there are only four edges, the overall the effect is no worse than a factor $(1+O(td_1^2/M))^{4}=1+O(td_1^2/M)$.

  Since it avoids edges of $H_t$, the \dsp\ move has probability at most
  $O(td_1^2\mu)/M$ of breaking any triangle.
  \item\label{case_e} We need to ensure that the edges $jk$ and $\ell m$ are absent in almost all $P_5$'s, so that they give a valid \dsp es.  We claim that these give rise to a $1+O(d_1^2/M)$ error factor. Consider $jk$. The expected number of $P_5$'s  in which it is present is at most
  \begin{equation*}
   \sum_{i,j,k,\ell,m=1}^n\frac{d_i(d_i-1)d_j(d_j-1)(d_j-2)d_k(d_k-1)(d_k-2)d_{\ell}d_m}{M^5}\leq \frac{d_1^2M_2^3M^2}{M^5}=O(\mu d_1^2),
  \end{equation*}
  and comparison with \eqref{prob:eq30} gives the claim. The calculation for $\ell m$ is similar.
  \item\label{case_f} We wish to ensure that the edges $jk$ and $\ell m$ are not edges of $H_t$. Otherwise, the $P_5$ will not give a valid \dsp.
  If $jk\in H_t$, then~\eqref{prob:eq30} becomes
  \begin{align*}
  2\sum_{jk\in E(H_t)} \sum_{i,\ell,m=1}^n & \frac{d_i(d_i-1)(d_j-2)(d_k-2)d_{\ell}d_m}{M^4}\leq\frac{2td_1^2 M_2M^2}{M^4}
   = \frac{2td_1^2 M_2}{M^2}\\ &= \frac{12td_1^2 \mu}{M} = O(td_1^2\mu/M).
  \end{align*}
  and, if $\ell m\in E(H_t)$,~\eqref{prob:eq30} becomes
  \begin{align*}
  2\sum_{\ell m\in E(H_t)}\sum_{i,j,k=1}^n & \frac{d_i(d_i-1)d_j(d_j-2)d_k(d_k-1)(d_\ell-1)(d_m-1)}{M^4} \leq\frac{2td_1^2 M_2^3}{M^4}\\
  & = \frac{12td_1^2\mu}{M} =  O(td_1^2\mu/M).
  \end{align*}
  Note that these events will not even occur asymptotically if $td_1^2\mu=o(M)$, since their expectation is then $o(1)$.
  \item\label{case_g} We must ensure that almost all $P_5$'s create only one triangle. The expected number in which the new edge $jk$ will create a second
   triangle $(hjk)$ is at most
  \begin{align*}
  \sum_{h,i,j,k,\ell,m=1}^n&\frac{d_h(d_h-1)d_i(d_i-1)d_j(d_j-1)(d_j-2)d_k(d_k-1)(d_k-2)d_{\ell}d_m}{M^6}\\
  & \leq \frac{d_1^2M_2^4M^2}{M^6} = \frac{6d_1^2\mu M_2}{M} = O(\mu d_1^3).
 \end{align*}
 The expected number of $P_5$'s in which the new edge $\ell m$ will create any triangle $(h\ell m)$ is at most
 \begin{align*}
  \sum_{h,i,j,k,\ell,m=1}^n&\frac{d_i(d_i-1)d_j(d_j-1)d_k(d_k-1)d_\ell(d_\ell-1)d_h(d_h-1)d_m(d_m-1)}{M^6}\\
  & = \frac{M_2^6}{M^6} = \mu^2 = O(\mu d_1^3).
 \end{align*}
 Combining these two cases gives an error factor $1+O(d_1^3/M)$.
\end{enumerate}
The effect of combining all these cases gives an overall error factor which we may bound by $1+O(d_1^2 (d_1+t)/M)=1+O(td_1^3/M)$, the critical cases arising from~\ref{case_d} and~\ref{case_g}.
Note also that~\ref{case_d} implies that almost all valid \dsp es give rise to a triangle in $G_{t+1}$ which is edge-disjoint from $H_t$.
\end{proof}

We now consider the \dsm es reversing these \dsp es.

\begin{lem}\label{prob:lem20}
Suppose that $M_2\geq M$.  If $td_1^2=o(M)$, then 
\[ \Expn[s_{t+1}^-\mid G_t]\,=\,(1+O(td_1^3/M))6(t+1)M.\]
\end{lem}
\begin{proof}
  We must choose a triangle $(ijk)$ and a labelling of it as $i,j,k$, which can be done in $6$ ways.
  Then, from the proof of Lemma~\ref{prob:lem10} we may assume that the triangle $(ijk)$ is edge-disjoint from any other  since, from case~\ref{case_d} in the proof of Lemma~\ref{prob:lem10}, this will be true for all but a proportion $O(td_1^2/M)$.
  So deleting $jk$ will delete exactly one triangle. Note that we could show that $(ijk)$ will be vertex-disjoint from $H_t$, but
  edge-disjointness is sufficient here.

  Deleting $\ell m$ will delete no triangles if $\ell m \notin E(H_t)$. So there are at most $M$, and at least $M-2|E(H_t)|$ choices for $\ell m$.
  So there are $(1-O(t/M))M$ choices. If $\ell j\in E(H_t)$ or $mk\in E(H_t)$, the \dsm\ will be invalid, but again case~\ref{case_d} in the proof
   of Lemma~\ref{prob:lem20} implies that only a proportion $O(td_1^2/M)$ of \dsm es will $G_{t+1}$ can be invalid for this reason. Otherwise, the expected number of triangles $(\ell jh)$ created by inserting $\ell j$ will be at most
    \begin{equation*}
  \sum_{h=1}^n\frac{d_h(d_h-1)(d_j-2)(d_\ell-1)}{M^2}\leq \frac{d_1^2 M_2}{M^2} \leq \frac{d_1^3}{M}.
  \end{equation*}
  So the error factor in $\Expn[s_{t+1}^-\mid G_t]$ is $1+O(d_1^3/M)$. The same calculation for $mk$ gives the same error,
  so the total effect will remain $1+O(d_1^3/M)$.

  Combining the error estimates we have an overall $1+O(td_1^3/M)$ factor, completing the proof.
\end{proof}

We can now prove Theorem~\ref{prob:thm10}. 

\begin{proof}[Proof of Theorem~\ref{prob:thm10}]
It follows from Lemmas~\ref{prob:lem10} and~\ref{prob:lem20}  and \eqref{prob:eq20} that
\begin{equation*}
  \frac{N_{t+1}}{N_t}=(1+O(td_1^3/M))\frac{\mu}{t+1}.
\end{equation*}
Iterating this through $t, t-1,\ldots, 1$,
\begin{equation}\label{prob:eq40}
 \frac{\Pr(t(G)=t)}{\Pr(t(G)=0)}= \frac{N_{t}}{N_0}=(1+O(td_1^3/M))^t\frac{\mu^t}{t!} = (1+O(t^2d_1^3/M))\frac{\mu^t}{t!}\sim \frac{\mu^t}{t!},
\end{equation}
if $t^2d_1^3=o(M)$, which is $t=o(\sqrt{M/d_1^3}$), so is valid for $t\leq t_0$.
Summing~\eqref{prob:eq40}
\begin{equation}\label{prob:eq50}
  \frac{\Pr(t\leq t_0)}{\Pr(t=0)}\sim \sum_{t=0}^{t_0}\frac{\mu^t}{t!}=e^\mu - \sum_{t=t_0+1}^{\infty}\frac{\mu^t}{t!}.
\end{equation}
Now $d_1^3=o(t_0)$ implies $\mu=o(t_0)$, so the term on the right is
\begin{equation*}
  \sum_{t=t_0+1}^{\infty}\frac{\mu^t}{t!}< \sum_{t=t_0+1}^{\infty}\left(\frac{o(t_0)}{t_0}\right)^t\sim o(1)^{t_0}=o(1).
\end{equation*}
In Section~\ref{sec:large_t} below, we show that $\Pr(t\leq t_0)=1-o(1)$. So, since $\mu=\Omega(1)$, \eqref{prob:eq50} gives
\begin{equation}\label{prob:eq60}
  \frac{\Pr(t\leq t_0)}{\Pr(t=0)}=\frac{1-o(1)}{\Pr(t=0)}\sim e^\mu - o(1)= (1-o(1))e^\mu.
\end{equation}
From \eqref{prob:eq60}, $\Pr(t=0)\sim e^{-\mu}$, and so \eqref{prob:eq40} gives
 $\Pr(t(G)=t)\sim e^{-\mu}\mu^t/t!$, and so $t(G)$ is approximately Poiss($\mu$).
\end{proof}

We state a useful special case.

\begin{cor}\label{cor:thm10-special}
Let $G$ be chosen uniformly from $\Gnd$ and let $\mu=\mu(\dsq)$.
Suppose that $M_2\geq M$.
If $d_1=o(n^{1/9})$ and $d_n\geq 1$ then $\Pr(t(G)=t)\sim e^{-\mu}\mu^t/t!$ for all $t=o(n^{1/3})$. 
\end{cor}

\begin{proof}  
Taking the $\omega(n)$ function in Theorem~\ref{prob:thm10} to equal $n^{1/3}/d_1^3$, the given 
assumptions imply that $\omega(n)\to\infty$ and that the function $t_0$ defined in
the statement of Theorem~\ref{prob:thm10} satisfies $t_0\geq n^{1/3}$.  The result now follows
follows from Theorem~\ref{prob:thm10}.
\end{proof}

\begin{rem}
Gao~\cite{Gao} has shown a normal limit distribution for $t(G)$ for $d$-regular graphs as $n\to\infty$, in the form $(t-\mu)/\sqrt{\mu}\stackrel{d}{\to} \mathrm{N}(0,1)$, provided $d\to\infty$ and $d=O(\sqrt{n})$. This is a limit for \pois{\mu} as $\mu\to\infty$, but whether Poisson approximation holds for such large $d$ is an open question.
\end{rem}

\subsection{Large number of triangles}\label{sec:large_t}

Here we will establish much weaker bounds, since our goal is only to show that the tail of the distribution of $t(G)$ has very small probability in the equilibrium distribution, 
provided $\lambda$ and $d_1$ are small enough. Specifically, we prove the following. 
Our argument follows the approach of the general-purpose result of Hasheminezhad and McKay~\cite[Corollary~1]{HashMc}, but to make our proof self-contained we give the details here.

\begin{thm}\label{thm:uppertail}
Let $t_0 = \Omega(n^{1/3})$ and suppose that $d_1 = o(n^{1/9})$.  
Let $\lambda\geq 1$ be constant.
If $\lambda >1$ then we also assume that $d_1\log\lambda = o(\log n)$. 
If $G$ is chosen from $\Gnd$ using the distribution $\pi_\lambda$ then 
\[
  \textstyle{\Pr_{\pi_{\lambda}}}(t\geq t_0) = e^{-\Omega(n^{1/4})}.\]
  Here the implicit constant inside the $\Omega(\cdot)$ term is positive.
\end{thm}

\begin{proof}
Given $G\in \cN_t$, let $s^-(G)$ be the number of 
valid \dsm es from $G\in \cG_i$ which create no triangle.
A lower bound on $s^-(G)$ 
is given by taking any triangle and any edge at distance at least 3 from the triangle.
There are at least $M-6(d_1-1)^3$ choices for this directed edge. If $d_1=o(M^{1/3})$,
this is $(1-o(1))M$. So $s^-(G)\geq (6-o(1))tM$ since we must choose a triangle edge and a direction,
and $s^-(\cN_t)=\sum_{G\in\cN_t}s^-(G)\geq (6-o(1))tMN_t$. No triangles are created by any of these \dsm es, but more than one triangle
may be destroyed by removal of the switch edges. So we destroy at least 1
and at most $2(d_1-1)$ triangles. These bounds are tight, since every edge might lie in a $(d_1+1)$-clique,
as is the case for graphs comprising $n/(d_1+1)$ cliques of size $(d_1+1)$.

So all the \dsm es from $G\in\cN_t$ lead to a graph in $\cN'_t=\bigcup_{i=1}^{2(d_1-1)}\cN_{t-i}$.
Suppose a proportion $\gamma_i$ lead to $\cN_{t-i}$, so $\sum_{i=1}^{2(d_1-1)}\gamma_i=1$, and note that hence 
\[ \sum_{i=1}^{2(d_1-1)}\gamma_iN_{t-i}\leq \max \{N_{t-i}:1\leq i\leq 2(d_1-1)\}=N_{t'}\] 
for some $t-2(d_1-1)\leq t' \leq t-1$.
Observe that we will have $t'=t-2(d_1-1)$ if $G$ is $d_1$-regular and $t=nd_1(d_1-1)/6$, so the lower bound can be tight.
Moreover, van der Hoorn et al.~\cite{HLM} show that almost all $d_1$-regular graphs have almost all their triangles in large cliques for  $t(G)=\Omega(n)$, so the analysis we give here may be difficult to improve significantly.

An upper bound on the number $s^+(G')$ of valid \dsp es from any $G'\in\cN'_t$ is the total number of possible switches,
valid or not. There are at most $\sum_{i=1}^n 2\binom{d_i}{2}(d_1-1)^2= M_2(d_1-1)^2$ possibilities,
so there are at most $M_2(d_1-1)^2 N_{t-i}$ valid \dsp es from 
$\cN_{t-i}$ to $\cN_t$.
Consequently, we must have 
\[ \frac{\gamma_i N_t}{N_{t'}} \leq \frac{\gamma_i N_t}{N_{t-i}} \leq (1+o(1)) \frac{M_2(d_1-1)^2}{6tM},\]
and summing over $i$ gives
\[\frac{N_t}{N_{t'}} \leq (1+o(1)) \frac{M_2(d_1-1)^2}{6tM}\leq \frac{\hat{\mu}}{t},\]
where $\hat{\mu}=(d_1-1)^3/6\geq M_2^3/M^3=\mu$, since $M_2\leq M(d_1-1)$.
Note that the final inequality is an equality if $G$ is $d_1$-regular.

We can iterate this to give
\[ \frac{N_t}{N_{0}}\leq \frac{\hat{\mu}^\ell}{tt'\cdots t^{(\ell-1)}},\]
where $t^{(j-1)}-t^{(j)} \leq 2d_1 -2$ for $0< j <\ell$, $t^{(0)}=t$, $t^{(\ell)}=0$.
The weakest bound results when $t^{(j-1)}-t^{(j)} = 2d_1 -2$ and $\ell= t/(2d_1-2)$, giving
\[ \frac{N_t}{N_0}\leq \frac{\hat{\mu}^{t/(2d_1-2)}}{(2d_1-2)^{t/(2d_1-2)}(t/(2d_1-2))!}
= \frac{(\hat{\mu}/(2d_1-2))^{t/(2d_1-2)}}{(t/(2d_1-2))!}< \bigg(\frac{e\hat{\mu}}{t}\bigg)^{t/(2d_1-2)},\]
using the bound $(t/(2d_1-2))!> (e^{-1}t/(2d_1-2))^{t/(2d_1-2)}$ from Stirling's approximation.
It follows that, when $G$ is chosen uniformly at random from $\Gnd$,
\[ \Pr(t(G)=t)  = \frac{N_t}{|\Gnd|} < \bigg(\frac{e\hat{\mu}}{t}\bigg)^{t/(2d_1-2)}.\]
If $\lambda > 1$ then we have assumed that $d_1\log\lambda = o(\log n)$,
and hence
$\lambda^{2d_1-2} \hat{\mu}=o(t_0)$. It follows for $\lambda\geq 1$ that 
$\lambda^{2d_1-2} \hat{\mu}\leq t_0/e^2$. 
Then under the Gibbs distribution $\pi_\lambda$,
\begin{align*}
  {\textstyle{\Pr_{\pi_\lambda}}}(t\geq t_0)\ 
  = \sum_{t=t_0}^{M_2/6}\frac{\lambda^t\, N_t}{Z_\lambda(\dsq)}
  &\leq 
  \sum_{t=t_0}^{M_2/6} \frac{\lambda^t\, N_t}{|\Gnd|}
  \\
  < &\ \sum_{t=t_0}^{M_2/6}\lambda^t\bigg(\frac{e\hat{\mu}}{t}\bigg)^{t/(2d_1-2)}
  <\ \sum_{t=t_0}^\infty\bigg(\frac{e\lambda^{2d_1-2}\hat{\mu}}{t}\bigg)^{t/(2d_1-2)}\\
  < &\ \sum_{t=t_0}^\infty e^{-t/(2d_1-2)}\ =\ e^{-t_0/(2d_1-2)}/\big(1-e^{-1/(2d_1-2)}\big)\\
  < &\ 2d_1e^{-t_0/(2d_1-2)}\ =\ e^{-\Omega(n^{1/4})},
\end{align*}
using the inequality $1-x < e^{-x}$ for $x>0$.   This concludes the proof.
\end{proof}

We can now prove Theorem~\ref{thm:asympt-Poisson}.

\begin{proof}[Proof of Theorem~\ref{thm:asympt-Poisson}]
Part (i) follows directly from Lemma~\ref{lem:asympt} using Corollary~\ref{cor:thm10-special} and Theorem~\ref{thm:uppertail} with $\lambda=1$.

For (ii), we must look at the distribution of $t(G)$ when $G$ is chosen from
$\Gnd$ using the Gibbs distribution $\pi_\lambda$.  Let $\unif$ denote the
uniform distribution on $\Gnd$.  Using Corollary~\ref{cor:thm10-special} and
Theorem~\ref{thm:uppertail}, 
\[
  \Ex{\unif}[\lambda^{t}]
    \sim \sum_{k\geq 0} \lambda^k e^{-\mu}\frac{\mu^k}{k!}\
    \sim e^{(\lambda-1)\mu}\, \sum_{k\geq 0} \frac{(\mu \lambda)^k}{k!} e^{-\mu\lambda}
    = e^{(\lambda-1)\mu}.
    \]
Hence, for the Gibbs distribution $\pi_\lambda$, if $s=o(n^{1/3})$ then
\[ {\textstyle{{\Pr}_{\pi_\lambda}}}(t(G)=s)=\frac{\lambda^s\Pr_{\mathrm{\unif}}(t(G)=s)}{ \Ex{\unif}[\lambda^{t}]}\sim\frac{\lambda^se^{-\mu}\mu^s/s!}{ e^{(\lambda-1)\mu}}
= \frac{e^{-\lambda\mu}(\lambda\mu)^s}{s!}.\]
Combining this with Theorem~\ref{thm:uppertail}
shows that $t(G)\distapp\pois{\lambda\mu}$ under $\pi_\lambda$, completing
the proof of part (ii). 
\end{proof}

\begin{rem}
\label{rem:indistinguishable}
If the assumptions of Theorem~\ref{thm:uppertail}(ii) hold and
$G$ is a random element of $\Gnd$ drawn from the distribution $\pi_\lambda$,
the probability that $t(G)>\nu$ is $o(1/q(n))$ for any polynomial $q(n)$.
Hence, after a ``burn-in'' period, if the modified Metropolis switch chain (Figure~\ref{fig:switch-metropolis}) is close
to equilibrium and we observe it for polynomially many steps, it should be
indistinguisable from the (unmodified) Metropolis switch chain (Figure~\ref{fig:switch-metropolis-unmodified}).  The same is true for the modified Metropolis \ds\ chain (Figure~\ref{fig:modified-triangle-switch-code}), compared with the unmodified Metropolis \ds\ switch chain. (The transition procedure of the unmodified Metropolis \ds\ chain is obtained from Figure~\ref{fig:modified-triangle-switch-code} by replacing $\min\{ t(H), \nu\}$ by $t(H)$ and replacing $\min\{ t(G),\nu\}$ by $t(G)$.)
\end{rem}

\subsection{Assumptions on the degree sequence}\label{s:assumptions}

In our asymptotic results we made the assumption that $M_2\geq M$.
We will now prove that this is equivalent to having average degree $\bar{d}\geq 2$.
For our Markov chain analysis the condition $\bard\geq 2$ is automatically satisfied, since we assume that $d_n\geq 3$ in Theorem~\ref{thm:rapid}. 

\begin{lem}\label{lem:dbar}
$M_2 \geq M$ \ifff $\bar{d}\geq 2$.
\end{lem}
\begin{proof} We have
  \[ 0 \leq \sum_{i=1}^n(d_i-\bard)^2=\sum_{i=1}^n d_i^2-n\bard^2=M_2+M-M\bard=M_2-M(\bard-1),\]
  so $M_2 \geq M$ if $\bar{d}\geq 2$.
  For the converse, suppose that $M_2 < M$. Let $n_j$ be the number of vertices
 with degree $j$, for all $j\geq 1$.  Then
 \[ M_2 - M = - n_1 + \sum_{j\geq 3} j(j-2)n_j < 0,\]
 and
 \[ M =  n_1 + 2 n_2 + \sum_{j\geq 3} j n_j
    \leq n_1 + 2 n_2 + \sum_{j\geq 3} j(j-2) n_j
       <  2n_1 + 2n_2 \leq 2n.
       \]
  So $\bard = M/n <2$,  completing the proof.
\end{proof}

The excluded graphs are rather uninteresting. They have a preponderance of components which are paths, as we see from the following.
\begin{lem}\label{lem:paths}
$M_2 \geq M$ for any connected graph $G$ which is not a path.
\end{lem}
\begin{proof}
If $G$ is not a path or a cycle, then $d_1\geq 3$. A cycle has $\bard=2$, so $M_2\geq M$ from Lemma~\ref{lem:dbar}. Otherwise, we can construct $G$ from a 3-star, a graph with degree sequence $(3,1,1,1)$, by successively adding edges and vertices. For a 3-star, $M_2=M=6$. Now, if we add an edge between vertices of degrees $d_i$ and $d_j$, $M\gets M+2$, and
\[ M_2\gets M_2+(d_i+1)d_i+(d_j+1)d_j-d_i(d_i-1)-d_j(d_j-1)=M_2+2(d_i+d_j)\geq M_2+2,\]
with equality only if we add an edge from a vertex of degree $d_i=1$ to a new vertex of degree 0.
Otherwise the inequality is strict. The claim now follows by induction.
\end{proof}
If $G$ is a path $P_\ell$ with $\ell$ vertices, then $M=2(\ell-1)$ and $M_2=2(\ell-2)$, so $M-M_2=2$.
Thus, if $M>M_2$, then Lemma~\ref{lem:paths} implies that most components must be paths.


\section*{Acknowledgements}
The authors thank the anonymous referee for their helpful comments, and
thank Brendan McKay for useful discussions.
Colin Cooper's research was supported by DFG Project~491453517. 
Martin Dyer's research was supported by an EPSRC grant EP/S016562/1 ``Sampling in hereditary classes''.
Catherine Greenhill's research was supported by Australian Research Council Discovery Project DP250101611.

\end{document}